\documentclass{article}
\usepackage{amssymb,amsmath,amsthm,graphicx,amsfonts,color}
\usepackage[english]{babel}
\usepackage{authblk, relsize, hyperref}
\newtheorem{theorem}{Theorem}
\newtheorem{corollary}[theorem]{Corollary}
\newtheorem{definition}[theorem]{Definition}
\newtheorem{example}[theorem]{Example}
\newtheorem{lemma}[theorem]{Lemma}

\newtheorem{proposition}[theorem]{Proposition}
\newtheorem{remark}[theorem]{Remark}

\title{An iterative process for approximating subactions}

\author{ Hermes H. Ferreira, Artur O. Lopes and Elismar R. Oliveira\\  Instituto de Matem\'atica e Estat\'istica - UFRGS, Brazil \\ }

\begin{document}

\maketitle

\begin{abstract}
We describe a procedure based on the iteration of  an initial function  by an appropriated operator, acting on continuous functions,  in order to get  a fixed point. This  fixed point  will be a calibrated subaction for the doubling map on the circle and a fixed Lipschitz potential. We study analytical and generic properties of this process and we provide some computational evaluations of subactions using a discretization of the circle.  The fixed point is unique if the maximizing probability is unique. We proceed a careful analysis of the dynamics of this operator close by the fixed point in order to explain the difficulty in estimating its asymptotic behavior.  We will show that  the  convergence rate can be in some moments  like $1/2$  and sometimes arbitrarily close to $1$.
\end{abstract}
	
\section{Introduction} \label{int}
	
Here we analyze  some properties of an iterative process (applied to an initial function)  designed for approximating subactions. Properties for a general form of such kind of algorithm were considered in \cite{Kra}, \cite{Ma1},  \cite{Dot2} and \cite{Ishi} (see also \cite{Bachar} for more recent results).  We analyze here the performance of a specific version of the algorithm which is useful in Ergodic Optimization.

 In a companion paper \cite{FLO} we will consider several examples. The sharp numerical evidence obtained from the algorithm permits to guess  explicit expressions for the subaction.

\medskip

We identify  $\mathbb{R}/\mathbb{Z}$ as $S^1$ and $T(x)=2x$ the doubling map. We denote by $\tau_i(x)=\frac{1}{2} (x+i-1), \; i=1,2$  the two inverse branches of $T$.
	\medskip

	\begin{definition}\label{max} Given a continuous function $A:S^1\to \mathbb{R}$ (or, $A:[0,1]\to \mathbb{R}$) we denote by
		$$m(A)= \sup_{ \rho\, \text{invariant for}\, T} \int A \, d \rho.$$	
		Any invariant probability $\mu$ attaining such
		supremum is called a {\bf  maximizing probability}.
	\end{definition}

The properties of the maximizing probabilities $\mu$ are the main interest  of Ergodic Optimization (see \cite{BLL}, \cite{G1}, \cite{CLT},\cite{J2}, \cite{J3} \cite{J4}, \cite{J6})

In Statistical Mechanics the limits of equilibrium probabilities when temperature goes to zero (see \cite{BLL})  are called ground states (they are maximizing probabilities).

\smallskip

A interesting line of reasoning is the following: there is a theory, someone gives a particular example which  leads to a problem to solve, then, use the theory to exhibit the solution. Is there a general procedure to find the solution of  this kind of problem? Here we will address this kind of query on the present setting.

\smallskip

	\begin{definition}\label{sub1} Given the Lipschitz continuous function $A:S^1\to \mathbb{R}$ the union of the supports of all the maximizing probabilities  is called the {\bf Mather set} for $A$.
		
	\end{definition}

	We will assume from now on that $A$ is Lipschitz continuous and that the maximizing probability is unique.

	It is known that for a generic Lipschitz potential $A$ (in the Lipschitz norm) the maximizing probability is unique and has support on a $T$-periodic orbit (see \cite{Con} and \cite{CLT}).
We do not have to assume here that the unique maximizing probability has support on a unique  periodic orbit.

	\begin{definition}\label{sub} Given the Lipschitz continuous function $A:S^1\to \mathbb{R}$, then
		a continuous function $u: S^1 \to \mathbb{R} $ is called a
		{\bf calibrated subaction} for $A$, if, for any $x\in S^1$, we have
		
		\begin{equation}\label{c} u(x)=\max_{T(y)=x} [A(y)+ u(y)-m(A)].\end{equation}
		
	\end{definition}
	
	Note that if $u$ is a calibrated subaction for $A$ then $u$ plus a constant is also a calibrated subaction for $A$.

For  Lipschitz potentials $A$ there exists Lipschitz calibrated subactions (see \cite{CLT}, \cite{Bou1}).
If the {\bf maximizing probability is  unique} (our assumption) then the {\bf  calibrated subaction is unique} up to adding a constant (see \cite{CLT} or \cite{GL1}).

Calibrated subactions play an important role in Ergodic Optimization (see \cite{BLL}, \cite{Sav} and \cite{G1}). From an explicit calibrated subaction one can guess where is the support of the maximizing probability. Indeed, given $u$ we have that for all
$x \in S^1$.
 \begin{equation}\label{x1} R(x)\,:=\, u(T(x)) -  u(x) - A(x) + m(A) \geq 0,\end{equation}
 and, for any point $x$ in the Mather set $R(x)=0.$ Moreover, if an invariant probability has support inside the set of points where $R=0$, then, this probability is maximizing (see \cite{CLT}).

In \cite{BLM} it is presented explicit expressions for the subaction in some nontrivial cases.

\begin{example}  \label{ex} We show in Figure  \ref{fig:jher}  the  graph of a potential  $A$, the graph of  the calibrated subaction $u$ and the graph of $R$.
The potential $A$ is zero at the points $1/4,3/4$ and it is equal to $-1$ in the points $0,1/2,1$.  The   set  $\{1/3,2/3\}$ is contained on the Mather set (then, it is the support of a maximizing probability)  and $m(A)= - \,1/3.$
  The calibrated subaction is $0$ at the point $1/2$ and equal to $2/3$ at the points $0,1$. The function $R$ is equal to $2/3$ at the points $0,1$ and it is equal to zero on the interval $[1/4,3/4]$. We point out that we easily guessed the explicit expression for the subaction $u$ from the picture obtained from the application  of the algorithm on the initial condition $f_0=0$.

\end{example}

\begin{figure}[h]
  \centering
  \includegraphics[scale=0.37]{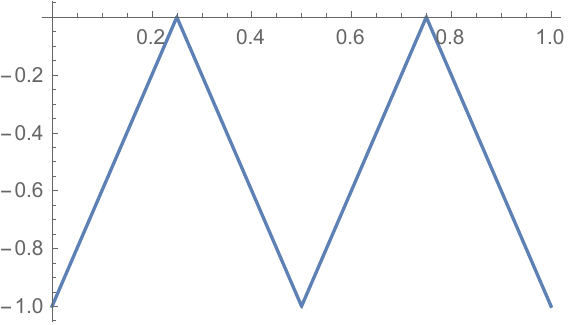}\qquad
  \includegraphics[scale=0.37]{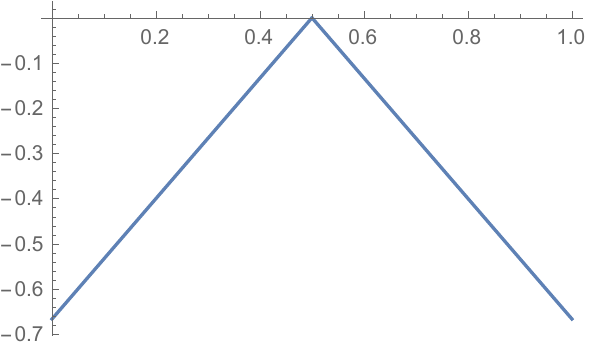}\qquad
  \includegraphics[scale=0.37]{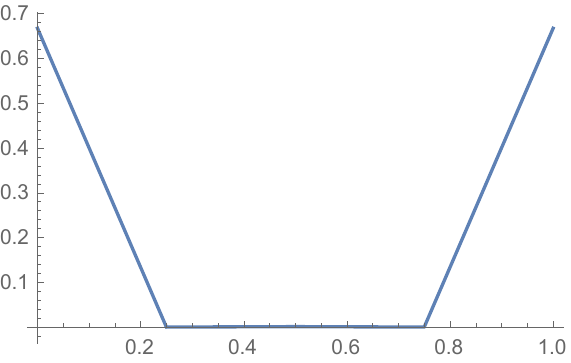}
  \caption{ From left to right: the  graph of the potential  $A$, the graph of  the calibrated subaction $u$ and the graph of $R$. }\label{fig:jher}
\end{figure}

	Given $x$, then, $u(x) = A (\tau_j (x)) + u (\tau_j(x)) - m(A)$, for some $j=1,2$. We say that $\tau_j(x)$ is a {\bf  realizer} for $x$. There are some points $x$ that eventually get at the same time two realizers.
	
	\medskip
	
	We are interested in an iteration procedure for getting  a good approximation of the subaction
	in the case the maximizing probability is unique. As a byproduct we will also get the value $m(A)$. This will help to get $R$ (as above) and eventually to find the support of the maximizing probability.

	We will consider a map  $\mathcal{G}$ acting on functions  such that the subaction $u$ is the unique  fixed point (we will have to consider the action on continuous functions up to an additive constant).  Unfortunately, $\mathcal{G}$ is not a strong contraction but we know that $\lim_{n \to \infty} \mathcal{G}^n (f_0)=u$ (for any given $f_0$).
	The performance of the iteration procedure is quite good and one can get easily  nice approximations.
	
	We explore here in section \ref{gen} the generic point of view on the set of continuous functions.  Given a fixed Lipschitz potential $A$ we will show  generic properties for the iterative process acting on continuous functions. In this direction expression (\ref{iimp}) (\, and (\ref{iimpil})\,) in Theorem \ref{tyy}, Theorem \ref{ty},  Corollary \ref{cor}
	and also expression (\ref{equimp}) in Remark 3  will provide this, and, therefore justify the excellent performance one can observe for  the iterative process  which we will describe here.
	
	A natural question: when the calibrated subaction is unique  is there an uniform exponential speed of  approximation (or, something numerically good)  of the iteration $\mathcal{G}^n (f_0)$ to the subaction?
	At least close by the subaction?
	In section \ref{close} we present a very detailed analysis of the action of the map $\mathcal{G}$ close by the fixed point $u$ and we will show that this is not the case.
	We will consider in Example \ref{E3}  a case where   where $|\mathcal{G}(f_{\varepsilon}) -\mathcal{G}(u )| =      |f_{\varepsilon} -u |$ , $\varepsilon>0$, for $f_{\varepsilon}$  as close as you want to the calibrated subaction $u$. In the positive direction one can also show that close by $u$ there are other $g_\varepsilon$, $\varepsilon>0$, such that,
	$|\mathcal{G}(g_{\varepsilon}) -\mathcal{G}(u )| =  1/2\,    |g_{\varepsilon} -u |$
	(see Corollary \ref{ftp}).

\begin{remark}
We emphasize the fact that in our computational evaluations we are not going to consider numerical aspects of this iteration process as rate of convergence, complexity or comparative efficiency with respect to other numerical schemes. First  because it is not our goal and more important because, as we are going to prove,   there exist a generic obstruction to get an analytical precise estimate for the convergence nearby the fixed point. We will show (see section \ref{close})  that  the  convergence rate can be in some moments  like $1/2$ (at each iteration) and sometimes arbitrarily close to $1$ (at each iteration).

 For related numerical computations we refer the reader to \cite{DOS1} and \cite{DOS2}. In these two papers the authors  define a general rigorous approach to discretize points on an interval (considering a finite lattice of points)  and also to discretize the action of some operators similar to the ones we will consider here. The aim is to find controlled approximations of a fixed point function for this discretized operator acting on a discrete lattice. One could employ the same ideas here with the appropriate adaptation but this is not the purpose of the present paper.

\end{remark}	
	
{\bf One final comment}: there are two major settings that people  analyze questions in Ergodic Optimization: 1) when it is assumed the potential is just continuous, and, 2) when  it is assumed some regularity (as Lipschitz for instance) on the potential.
The two cases are conceptually distinct: in the first case, generically, the maximizing probability has support on the all space (see \cite{Bou2} and \cite{J2}) and in the second case, generically, the support has support on a periodic orbit (see \cite{Con}  and \cite{CLT}). In the first case, generically, subactions are of no help. It is in the second case that subactions are of great help for identifying the support of the maximizing probability. In our work we  introduce  a nice tool for  identifying, generically,   the maximizing probability (see \cite{FLO}).

\section{The 1/2 iterative procedure} \label{1/2}
	
	On the set of continuous functions $f:S^1 \to \mathbb{R}$ we consider
	the sup norm: $|f|_0= \sup \{|f(x)|, \, x \in S^1 \}$. This set is denoted by $C^0=C^0(S^1 , \mathbb{R})$.
		
	\begin{definition} In $C^0(S^1 , \mathbb{R})$ we consider the equivalence relation
		$f \sim g$, if $f-g$ is a constant.		The set of classes is denoted by  $\mathcal{C}= C^0 /\mathbb{R}$ and, by convention, we
		will consider in each class a representative which
		has supremum equal to zero.
	\end{definition}

	In $\mathcal{C}$ we consider the quotient  norm (see section 7.2 in \cite{Ma})
	$$ |f|= \inf_{\alpha \in \mathbb{R} } | f + \alpha|_0.$$
	We can also consider this norm $ |f| $ restricted to set of   Lipschitz functions in  $\mathcal{C}.$
	$(\mathcal{C}, |\,\cdot \,|)$  is a Banach space (see \cite{Ma}) .	
	 As $S^1$ is compact  we get that: for any given $f$ there exists $\alpha$, such that, $|f|= |f+\alpha|_0$.

	We denote sometimes the constant $\alpha$ associated to $f$ by $\alpha_f:=-\frac{\max f + \min f}{2}$.
	We point out that when we write $|f(x)|$ this means the modulus of an element in $\mathbb{R}$ and $|f|$ means the norm defined above.
		
		\medskip

	\begin{definition}\label{op1} Given a Lipschitz  continuous function $A:S^1\to \mathbb{R}$ we consider the operator (map) $\hat{\mathcal{L}}= \hat{\mathcal{L}}_A$, such that, for $f: S^1 \to \mathbb{R} $, we have   $\hat{\mathcal{L}}_A(f)=g$, if
		
		\begin{equation}\label{d1} \hat{\mathcal{L}}_A(f)(x)=   g(x)=\max_{T(y)=x} [A(y)+  f(y) - m(A)].\end{equation}
		
		for any $x\in S^1$.
		
	\end{definition}
	
	For the given Lipschitz  continuous function $A:S^1\to \mathbb{R}$ the operator  $\hat{\mathcal{L}}_A$ acts in $\mathcal{C}$ as well as in $\mathcal{C}_0$.

\begin{figure}[h] \label{fig:Graf}
\center
\includegraphics[height=2cm,width=4cm, trim=0.1in 0.08in 0.1in 0.08in, clip]{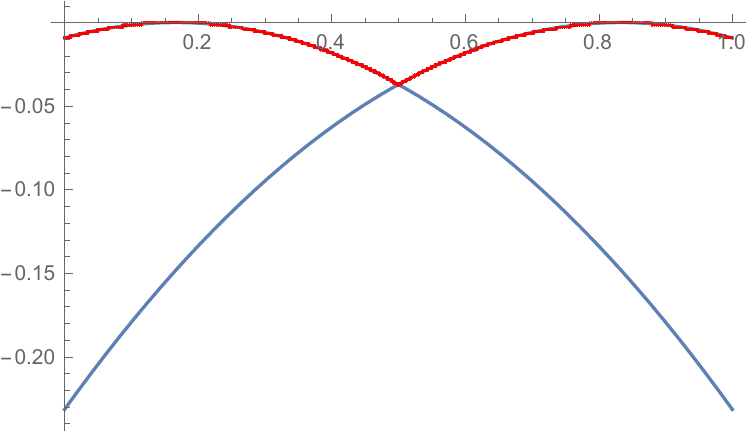}
	\caption{Case $A(x)= - (x-1/2)^2$  and $T(x)=- 2 x$ (mod 1) - In this case $m(A)= -1/36$. The red graph  describes the values of the approximation  (via $1/2$-algorithm) to the calibrated subaction $u$  given by $\mathcal{G}^{10}(0)$   (using the language $C^{++}$ and a mesh of points) and the two blue graphs  describe,  respectively, the graphs of $ x\, \to\, - 1/3 x^2 + 1/9 x ,\,\, $ and $x\,\to \, - 1/3 x^2 +5/9 x - 2/9$. The supremum of these two functions is the exact analytical expression for  the graph of the calibrated subaction $u$. The red color obliterates the blue color.}
	 \label{fig:Graf}
\end{figure}
	
	Note that $u$ is a fixed point  for such operator $f \to \hat{\mathcal{L}}_A(f)$, if and only if, $u$ is a calibrated subaction. It is well known there exists calibrated subactions when  $A$ is of Lipschitz  class (see for instance \cite{BLL}).

	One could hope that a high iterate $\hat{\mathcal{L}}_A^n(f_0)$ $(n$ large) would give an approximation of the calibrated subaction.
This operator will not be very helpful  because we have to known in advance the value $m(A)$. Even if we know the value $m(A)$ the iterations
	$\hat{\mathcal{L}}_A^n(f_0)$ applied  on an initial continuous function $f_0$ may not converge. This can happen even in the case the calibrated subaction is unique.

	\begin{definition}\label{op1} Given a Lipschitz continuous function $A:S^1\to \mathbb{R}$ we consider the operator (map) $\mathcal{L}= \mathcal{L}_A: \mathcal{C} \to \mathcal{C}$, such that, for $f: S^1 \to \mathbb{R} $, we have   $\mathcal{L}_A(f)=g$, if
		\begin{equation}\label{e1} \mathcal{L}_A(f)(x)=   g(x)=\max_{T(y)=x} [A(y)+  f(y)] - \sup_{s \in S^1} \{\max_{T(r)=s} [A(r)+  f(r)] \} .\end{equation}
		for any $x\in S^1$.
		
	\end{definition}

	The advantage here is that we do not have to know the value $m(A)$.
	In the same way as before $u$ is a fixed point  for the operator
	$\mathcal{L}_A(f)$, if and only if, $u$ is a calibrated subaction.

We call the iterative  procedure (defined below and denoted by $\mathcal{G}$)  the $1/2$-iterative process. It is a particular case of the iteration procedure described on \cite{Ma1} and \cite{Ishi}. From these two papers it follows that given any initial function $f_0\in \mathcal{C}$ we have that
$ \lim_{n \to \infty} \mathcal{G}^n (f_0)$ exists and it is the subaction $u$ (which belongs to $\mathcal{C}.$)

\medskip

{\bf Remark 2:}	The iterations
	$\mathcal{L}_A^n(f_0)$ applied  on an initial continuous function $f_0$ may not converge. This can happen even in the case the calibrated subaction is unique as some examples can show.  The bottom line is: we have to use $\mathcal{G}$ and not $\mathcal{L}_A$.
		$\hfill \diamondsuit$
	\medskip

	In order to show the power of the approximation scheme we consider an example where the subaction $u$ was already known. The dynamics is $T(x)=- 2 x$ (mod 1) (\,not $T(x)= 2 x$ (mod 1)\,). The $1/2$-algorithm works also fine in this case. According to example 5 in pages 366-367 in \cite{LOT} the subaction $u$ (see picture on page 367 in \cite{LOT}) for the potential $A(x)= - (x-1/2)^2$ is
$$u(x)=	\max \{\,  - 1/3 x^2 + 1/9 x ,\,\, - 1/3 x^2 +5/9 x - 2/9\,\}.$$

More generally, in page 391 in \cite{LOT} is described a natural procedure to get the subaction $u$ for potentials $A$ which are quadratic polynomials.
The maximizing probability $\mu$ in this case has support on the orbit of period two (according to \cite{J3}, \cite{J4} and \cite{J6}) and $m(A)=-1/36$.
One can see from Figure \ref{fig:Graf} a perfect match of the solution obtained from the algorithm described by $\mathcal{G}$ and the graph of the exact calibrated subaction $u$.

	\begin{definition}\label{op1}  Given a Lipschitz continuous function $A:S^1\to \mathbb{R}$ we consider the operator (map) $\mathcal{G}= \mathcal{G}_A: \mathcal{C} \to \mathcal{C}$, such that, for $f: S^1 \to \mathbb{R} $, we have   $\mathcal{G}_A(f)=g$, if
	$$\mathcal{G}_A(f)(x)=   g(x)=\frac{\max_{T(y)=x} [A(y)+  f(y)] + f(x)}{2} \,\,\,\,-c_f$$
	for any $x\in S^1$,
where
\begin{equation}\label{c1}c_f:= \sup_{s\in S^1} \frac{\max_{T(r)=s} [A(r)+  f(r)] + f(r)}{2} .\end{equation}	
\end{definition}
We will show later in Theorem \ref{ty1}  that $\,| \mathcal{G}(f)-  \mathcal{G}(g)|\leq |f-g|   $, for any $f,g \in \mathcal{C}.$ Therefore, $\mathcal{G}$ is Lipschitz continuous.

The operator $\mathcal{G}$ is not linear. As we already mentioned we called the procedure based on high iterations $ \mathcal{G}^n (f_0)$ the $1/2$ iterative procedure.

\medskip

The above Definition \ref{op1} was inspired by expressions (5.1) and (5.2) of \cite{Chou}. This is a particular case of a more general kind of numerical iteration procedure known as the Mann iterative process (see \cite{Ma1}, \cite{Dot2}, \cite{Kra}, \cite{Ishi} and \cite{Ma2}).

 Assuming that the subaction $u$ for the Lipschitz potential $A$ is unique (up to adding constants) it follows (as  particular case) from the general results  of W. Dotson, H. Senter and  S. Ishikawa (see   Corollary 1 in \cite{Dot2}, \cite{Ma1} or \cite{Ishi}) that  $\lim_{n \to \infty} \mathcal{G}^n (f_0)=u,$
 for any given $f_0\in \mathcal{C}$.

 The special
$\mathcal{G}$ presented above was not previously consider in the literature (as far as we know).
	
	Note that $\mathcal{G}_A(f+c)= \mathcal{G}_A(f)$ if $c$ is a constant and also that for any $f$ the supremum of  $\mathcal{G}_A(f)$ is equal to $0.$

When running the iteration procedure on a computer (using the language $ C^{++}$) one fix a mesh of points in $[0,1]$ and perform the operations on each site. The pictures we will show here are obtained in this way when we consider a large number of points equally spaced.

One important issue on the companion paper \cite{FLO} with explicit examples is corroboration. By this we mean:
we derive analytically some complicated expressions and we use the algorithm to compare
and confirm that our reasoning was correct.

The next proposition is a direct consequence fo the definition of $\mathcal{G}_A$ but we will present a proof for the benefit of the reader.

	\begin{proposition}
		
		If $u$ is such that $\mathcal{G}_A(u)=u$, then,
		$u$ is a calibrated subaction and
		\begin{equation}\label{c12} m(A)= \sup_{z} \max_{T(y)=z} [A(y)+  u(y)] + u(z) .\end{equation}
		
	\end{proposition}
	
	{\bf Proof:} If
	\begin{equation}\label{c13} u(x)=   \frac{\max_{T(y)=x} [A(y)+  u(y)] + u(x)}{2} -c_u ,\end{equation}
	then, for all $x$, we obtain
	$ u(x)= \frac{\max_{T(y)=x} [A(y)+  u(y)] + u(x)}{2} -c, $
	where $c=c_u=  \sup_{z} \frac{\max_{T(y)=z} [A(y)+  u(y)] + u(z)}{2} $ is  constant.
	This means that
	$$2 u(x)= \max_{T(y)=x} [A(y)+  u(y)]  +  u(x) - 2c, $$
	and, finally, we get
	$ u(x)= \max_{T(y)=x} [A(y)+  u(y)]   - 2c,$ for any $x$.
	
	In the end of the proof of Theorem 11 in \cite{BCLMS} it is shown that this implies that $m(A) = 2 c$ and it follows that $u$ is a calibrated subaction.
	
	\qed

\medskip

 {\bf Counter example 1:} 	$\mathcal{G}$ may not be a strong contraction (by a factor smaller than $1$).
 We will present an example where  $f_0,g_0\in \mathcal{C}$ but $ |\mathcal{G}(f_0)- \mathcal{G}(g_0)|=1/2=  |f_0-g_0|.$

 Consider the potential $A$ with the graph given by Figure  \ref{fig:h1}. This potential is linear by parts and has the value $0$  on the points $1/8, 1/4, 3/4, 7/8$. The value $-1$ is attained at the points $0, 3/16,1/2 ,13/16,1$.

 Denote $g_0=0$ and $f_0=A$. Then, $|f_0-g_0|=|f_0-g_0+1/2|_0=1/2$.
 We denote $f_1= \mathcal{G}(f_0) $ and  $g_1= \mathcal{G}(g_0) $.
 The graph of the function $x \to | f_1(x)- g_1(x)+0.5|$ is described by the bottom rigth picture on Figure \ref{fig:h1}.
 One can show that   $|f_1-g_1|=| f_1- g_1+1/2|_0=1/2$. Therefore,  for such potential $A$ the transformation
 $\mathcal{G}$ is not a strong contraction. Theorem \ref{ty1} shows that
 $\mathcal{G}$ is a weak contraction.

$\hfill \diamondsuit$
\medskip

\begin{figure}[h!]
  \centering
  \includegraphics[scale=0.35]{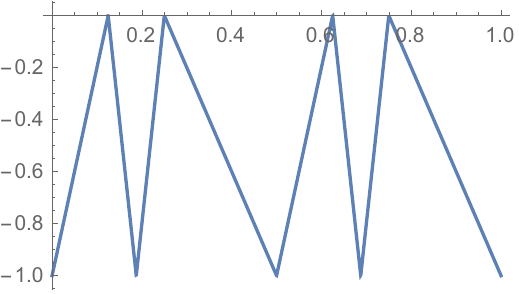}\qquad
  \includegraphics[scale=0.35]{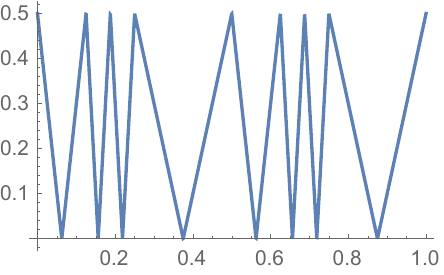}\qquad
  \includegraphics[scale=0.35]{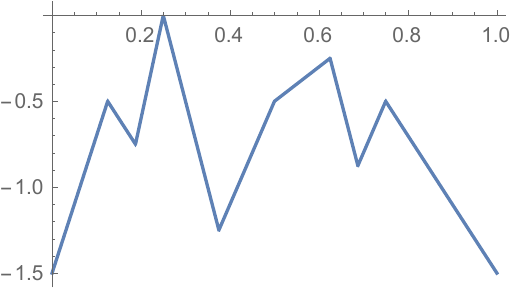}\qquad
  \includegraphics[scale=0.30]{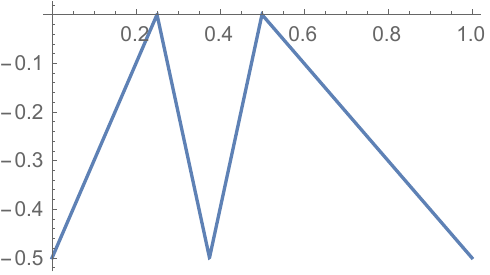}\qquad
   \includegraphics[scale=0.30]{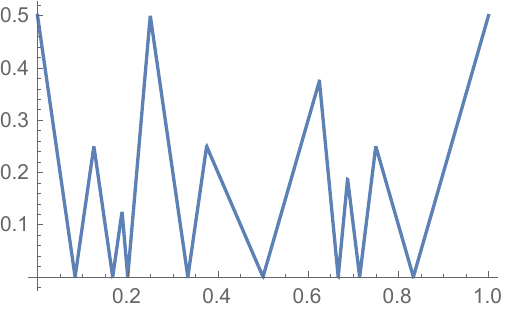}
  \caption{On the top: from  left to right  the graph of $A=f_0$, the graph of $x \to |(f_0(x)-0) +0.5|$,  the graph of   $f_1=\mathcal{G}(f_0)$. On the bottom: from left to right   the graph of $g_1=\mathcal{G}(0) =\mathcal{G}(g_0) $  and the graph of $x \to |f_1(x)-g_1( x)+0.5|.$ Therefore, $\mathcal{G}$ is not a strong contraction because $|f_0-g_0|= 1/2 = |f_1-g_1|=  | \mathcal{G}(f_0) - \mathcal{G}(g_0)|$. }\label{fig:h1}
\end{figure}

	For a fixed $K>0$ we denote by $\mathcal{C}_K$, the set of  Lipschitz functions $f:S^1 \to \mathbb{R}$ in $\mathcal{C}$,  with Lipschitz constant smaller or equal to $K$. By Arzela-Ascoli Theorem $\mathcal{C}_K$ is a compact space in $\mathcal{C}$.

	 \begin{theorem} Suppose $A$ has Lipschitz constant equal to $K$. Then, $\mathcal{G}(\mathcal{C}_K) \subset \mathcal{C}_K$. Therefore, the image of $\mathcal{C}_K$ by $\mathcal{G}$ is compact for the quotient norm in $\mathcal{C}$.

	 \end{theorem}
	 \medskip
	
	 {\bf Proof:} Denote $f_1= \mathcal{G}(f_0).$

	 Given a point $y$ assume without loss of generality that $f_1(x)-f_1(y)\geq 0.$
	
	 Then,
	 $$  f_1(x)-f_1(y)	\leq \left[\,\frac{ A(\tau_{a_0^{x,f_0}}(x))}{2} + \frac{1}{2}(\,f_0(\tau_{a_0^{x,f_0}}(x)) + f_0(x)\, )\,\right] -$$
	 $$\left[\,\frac{ A(\tau_{a_0^{x,f_0}}(y))}{2} + \frac{1}{2}(\,f_0(\tau_{a_0^{x,f_0}}(y)) + f_0(y)\, )\,\right] =$$
	 $$\frac{1}{2} [ A(\tau_{a_0^{x,f_0}}(x)) -  A(\tau_{a_0^{x,f_0}}(y))]+ 	 \frac{1}{2}[\,f_0(\tau_{a_0^{x,f_0}}(x))-\,f_0(\tau_{a_0^{x,f_0}}(y))]+\frac{1}{2} [f_0(x)-f_0(y)]\leq$$
	 $$ K\,\frac{1}{2}\, |\tau_{a_0^{x,f_0}}(x)) -  \tau_{a_0^{x,f_0}}(y)| + K\,\frac{1}{2}\, |\tau_{a_0^{x,f_0}}(x)) -  \tau_{a_0^{x,f_0}}(y)|+\frac{1}{2}K\, |x-y|=  $$
	 $$ K\,\frac{1}{2}\,\frac{1}{2} |x-y| +  K\,\frac{1}{2}\,\frac{1}{2} |x-y|   +\frac{1}{2}K\, |x-y|= K\,|x-y|. $$

	 \qed

\medskip

The next theorem is a direct consequence of the nonexpansiveness of $\mathcal{L}_A$ but we will present a proof for the benefit of the reader.
	
	\begin{theorem} \label{ty1}
	 	Given the functions $f,g \in \mathcal{C}$ we have
	 	$$ |\mathcal{G}(f)- \mathcal{G}(g)|\leq |f-g|.$$

	 \end{theorem}
	{\bf Proof:}  Let $[f], [g] \in \mathcal{C}$ and $d = \alpha_{f-g }\in \mathbb{R}$ such that
\[
|[f]- [g]|= |f -g +d |_{0}.
\]

We denote $k=\alpha_{\mathcal{G}(f)- \mathcal{G}(g)  } $ the value such that
$|\mathcal{G}([f])- \mathcal{G}([g])|=|\mathcal{G}([f])- \mathcal{G}([g]) +k|_0.$

In order to estimate $|\mathcal{G}([f])- \mathcal{G}([g])|$ consider
$\mathcal{G}(f)(x)- \mathcal{G}(g)(x)=
$
$$
-c_{f}+\frac{1}{2} f(x) + \frac{1}{2}\max_{i \in \{1,2\}}\left[ (A+f)(\tau_{i}(x)) \right] +c_{g}-\frac{1}{2} g(x) - \frac{1}{2}\max_{i \in \{1,2\}}\left[ (A+g)(\tau_{i}(x)) \right],
$$
which means
$
2\left(\mathcal{G}(f)(x)- \mathcal{G}(g)(x) +c_{f}-c_{g}\right)=
$
\[
 f(x)-g(x) + \max_{i \in \{1,2\}}\left[ (A+f)(\tau_{i}(x)) \right]  - \max_{i \in \{1,2\}}\left[ (A+g)(\tau_{i}(x)) \right].
\]

We add  $d$ to both sides obtaining
$$
2\left(\mathcal{G}(f)(x)- \mathcal{G}(g)(x) +c_{f}-c_{g}+ d\right)=$$
$$
 f(x)-g(x)+d + \max_{i \in \{1,2\}}\left[ (A+f+d)(\tau_{i}(x)) \right]  -\max_{i \in \{1,2\}}\left[ (A+g)(\tau_{i}(x)) \right],
$$
which can be rewritten as
$
2\left(\mathcal{G}(f)(x)- \mathcal{G}(g)(x) +c_{f}-c_{g}+ d\right)=$
$$
 \left(f(x)-g(x)+d\right) + \max_{i \in \{1,2\}}\left[ (A+g+f-g+d)(\tau_{i}(x)) \right]  - \max_{i \in \{1,2\}}\left[ (A+g)(\tau_{i}(x)) \right].
$$
We notice that
$ -|[f]- [g]| \leq f(y)-g(y)+d \leq |[f]- [g]|$
for any $y\in X$. By monotonicity of the supremum we get
$$
-|[f]- [g]| + \max_{i \in \{1,2\}}\left[ (A+g)(\tau_{i}(x))\right]  \leq$$
$$\max_{i \in \{1,2\}}\left[ (A+g+f-g+d)(\tau_{i}(x)) \right] \leq | [f]- [g]| + \max_{i \in \{1,2\}}\left[ (A+g)(\tau_{i}(x))\right],
$$
which  is equivalent to
$
-|[f]- [g]| \leq$
$$ \max_{i \in \{1,2\}}\left[ (A+g+f-g+d)(\tau_{i}(x)) \right] -\max_{i \in \{1,2\}}\left[ (A+g)(\tau_{i}(x))\right] \leq | [f]- [g]|,
$$
thus
\[
|\max_{i \in \{1,2\}}\left[ (A+g+f-g+d)(\tau_{i}(x)) \right]  - \max_{i \in \{1,2\}}\left[ (A+g)(\tau_{i}(x)) \right] |_{0} \leq | [f]- [g]|.
\]

We assumed  that
$|f-g+d |_{0} = | [f]- [g]|.
$
Therefore, using the
 two last inequalities we get
$
|2\left(\mathcal{G}(f)- \mathcal{G}(g)+c_{f}-c_{g}+ d\right)|_{0} \leq | [f]- [g]|+ | [f]- [g]|,
$
which is equivalent to
\begin{equation} \label{uru}
|\mathcal{G}(f)- \mathcal{G}(g) +(c_{f}-c_{g}+ d)|_{0} \leq | [f]- [g]|.
\end{equation}
We recall that
$
|\mathcal{G}([f])- \mathcal{G}([g]) |=$
$$\min_{k\in \mathbb{R}} |\mathcal{G}(f)- \mathcal{G}(g) + k |_{0} \leq |\mathcal{G}(f)- \mathcal{G}(g) +(c_{f}-c_{g}+ d)|_{0} \leq | [f]- [g]|,
$$
and this finish the proof.  \qed

	\medskip

	\section{Generic properties} \label{gen}
	
	We will show a generic property for the iterative process acting on continuous functions for a given fixed Lipschitz potential $A$.
	
	 \begin{definition} \label{dee}  Consider the set $\mathfrak{A}\subset \mathcal{C} \times \mathcal{C}$
	 	of pairs of functions $(f_0,g_0)$,
	 	such that,
	 	if  $ |f_0-g_0|=|(f_0- g_0 )+ \alpha_{f_0-g_0)}|_0=  (f_0-g_0) (r)+ \alpha_{f_0-g_0} $, for some $r$, then,
	 	$$(f_0-g_0) (r) \neq  (f_0-g_0) (\tau_1(r))\,\,\text{and}\,\, (f_0-g_0) (r)\neq (f_0-g_0) (\tau_2(r)).$$
	 	
	 \end{definition}

Note that the above condition does not depends on the potential $A$.
In the case $f_0(x)-g_0(x)+ \alpha_{(f_0-g_0)}$ attains the supremum in a unique point then $(f_0,g_0)\in\mathfrak{A}$. Obviously, we could choose $\mathfrak{A}\subset \mathcal{C}$ the set of $h \in  \mathcal{C}$  such that $h (r) \neq  h (\tau_1(r))\,\,\text{and}\,\, h (r)\neq h(\tau_2(r))$, but our choice $h=f_0-g_0$ avoid this relabeling in the future.

We will show in  Corollary  \ref{cor} that the condition $(f,g) \in  \mathfrak{A}$ is generic in  $\mathcal{C} \times \mathcal{C}$.

\medskip

	 \begin{theorem} \label{ty}
	 	Given the functions $f_0,g_0 \in \mathcal{C}$,
	 	assume $(f_0,g_0) \in\mathfrak{A} $.
	 	In this case, if
	 	$ |\mathcal{G}(f_0)- \mathcal{G}(g_0)|= |f_0-g_0|,$
	 	then, $f_0=g_0.$
	 	
\end{theorem}
{\bf Proof:}
We denote by  $d=\alpha_{f_0-g_0}$  the value such that
 $|  (f_0-g_0) +d |_0 =  |f_0-g_0|$.

We denote by  $z_0$ the point such that $ |f_0 - g_0|=| f_0(z_0) - g_0(z_0) +d|$.
Without loss of generality we assume that $  f_0(z_0) - g_0(z_0) +d>0. $

	Note that  $|(f_0-g_0+d)(z_0)|$ also maximizes
	\begin{equation} \label{pk} x \to  |(f_0-g_0+d)(x)|.
	\end{equation}
	
	  Note that $d$ was determined by the choice $(f_0-g_0)$ (and, not $(g_0-f_0)$).
	
We denote by $k=\alpha_{ \mathcal{G}(f_0)- \mathcal{G}(g_0)}$ the value   $|\mathcal{G}(f_0)- \mathcal{G}(g_0)| +k   |_0=|\mathcal{G}(f_0)- \mathcal{G}(g_0)|.$

 Assuming 	$|\mathcal{G}(f_0)- \mathcal{G}(g_0)|= |f_0-g_0|,$ then,
from (\ref{uru}) we get
 \begin{equation} \label{uru1}
|\mathcal{G}(f_0)- \mathcal{G}(g_0)| =
|\mathcal{G}(f_0)- \mathcal{G}(g_0) +k|_0\leq
|\mathcal{G}(f_0)- \mathcal{G}(g_0) +(c_{f_0}-c_{g_0}+ d)|_0 \leq | [f_0]- [g_0]|.
\end{equation}

 Therefore, $k$ can be taken as $ k= c_{f_0}-c_{g_0}+d.$ Note that $k$ was determined by $d$ and the choice $(f_0-g_0)$ (and, not $(g_0-f_0)$).

We denote by $z_1$ a point such that $ |\mathcal{G}(f_0)- \mathcal{G}(g_0)|=
 |\mathcal{G}(f_0)(z_1)- \mathcal{G}(g_0)(z_1) +k|=| f_1(z_1) - g_1(z_1) +k|$.

In the case $ (f_1-g_1)(z_1)+k\leq 0$  we know that there exists another
point $\tilde{z_1}$, such that,  $ 0 \leq (f_1-g_1)(\tilde{z_1})+k = |\mathcal{G}(f_0)- \mathcal{G}(g_0) +k|_0.$

Therefore, without loss of generality, we can always assume that it is true  $ (f_1-g_1)(z_1)+k\geq 0.$

 Assume that
$(f_0,g_0) \in\mathfrak{A} $.

Under the above conditions in $f_0,g_0$, there exists $z_0$, $z_1$, $\bar{z}=\tau_{a_0^{z_1,f_0}}(z_1)$ and $\bar{w}=\tau_{a_0^{z_1,g_0}}(z_1)$ such that
	  $$   (f_0-g_0)(z_0)  +d =  |f_0-g_0|=|\mathcal{G}(f_0)- \mathcal{G}(g_0)|= (f_1-g_1)(z_1)+k=$$
	  $$[\frac{ A(\bar{z})}{2} + \frac{1}{2}(f_0(\bar{z}) + f_0(z_1))]  -[\frac{ A(\bar{w})}{2} + \frac{1}{2}(g_0(\bar{w}) + g_0(z_1) )] +k-  c_f+c_g   \leq$$
	 $$[\frac{ A(\bar{z})}{2} + \frac{1}{2}(f_0(\bar{z}) + f_0(z_1))]  -[\frac{ A(\bar{z})}{2} + \frac{1}{2}(g_0(\bar{z}) + g_0(z_1) )]  +k-  c_f+c_g  =$$
	$$[ \frac{1}{2}(f_0(\bar{z}) + f_0(z_1))]  -[ \frac{1}{2}(g_0(\bar{z}) + g_0(z_1) )]+k-  c_f+c_g    =$$
		$$\frac{1}{2}(f_0(z_1) - g_0(z_1)) +  \frac{1}{2}(f_0(\bar{z}) - g_0(\bar{z} )+k-  c_f+c_g    =$$
	\begin{equation}
	\label{aqui} \frac{1}{2}(f_0(z_1) - g_0(z_1)) +  \frac{1}{2}(f_0(\bar{z}) - g_0(\bar{z} )+d.
	\end{equation}
	
As $(f_0-g_0+d)(z_0)>0$ is a supremum, it follows from the above that
	  $$  (f_0-g_0)(z_0)+ d\leq
	 \frac{1}{2}[ (f_0 - g_0)(z_1) +d ]   + \frac{1}{2}[ (f_0 - g_0 )(\bar{z})+d]     \leq$$
	\begin{equation} \label{rrr}
	 \frac{1}{2}[ (f_0 - g_0)(z_0) +d] +      \frac{1}{2} [(f_0 - g_0 )(z_0)+d]= (f_0 - g_0)(z_0) +d.
	\end{equation}

	  $ (f_0 - g_0)(z_1) +d$ and $  (f_0 - g_0 )(\bar{z})+d$ can not be both negative (because  $   (f_0-g_0)(z_0)  +d >0$).
	
	 Both  $ (f_0 - g_0)(z_1) +d$ and $  (f_0 - g_0 )(\bar{z})+d$ are positive. Otherwise,  from (\ref{rrr}) we get
	 $(f_0-g_0)(z_0)+ d< \,\frac{1}{2} [\, (f_0-g_0)(z_0)+ d\,].$
This implies that
	 $
	 \frac{1}{2}[ (f_0 - g_0)(z_1)+d]   + \frac{1}{2} [  (f_0 - g_0 )(\bar{z})+d]     = (f_0-g_0)(z_0) + d. $
	
	Remember that  $d=\alpha_{ f_0-g_0 }=  - \frac{\max (f_0-g_0) + \min(f_0-g_0)}{2}.$
From \ref{rrr} we get
 $   (f_0 - g_0)(z_1)+ d =(f_0-g_0)   (\bar{z})+ d   =(f_0 - g_0)(z_0) +d .$

As $(f_0,g_0) \in\mathfrak{A} $ we get by Corollary \ref{cor} a contradiction.

\medskip

\qed
	
	 \medskip
	
	{\bf Remark 3:}  Given the  point $z_1$ above  (supremum of
	 $x \to (f_1(x)-g_1(x))+k$) we get from (\ref{aqui})   that
	\begin{equation} \label{aqui1}(f_1-g_1)(z_1)+k\leq \frac{1}{2}(f_0(z_1) - g_0(z_1)+ d) +  \frac{1}{2}(f_0(\tau_{a_0^{z_1,f_0}}(z_1)) - g_0(\tau_{a_0^{z_1,f_0}}(z_1) +d) . \end{equation}
	Note that if
$ f_0(z_1) - g_0(z_1)+ d$
and
$ f_0(\tau_{a_0^{z_1,f_0}}(z_1)) - g_0(\tau_{a_0^{z_1,f_0}}(z_1) +d)$	have opposite signals, then we get a better rate
	
\begin{equation} \label{equimp} |\mathcal{G}(f_0)- \mathcal{G}(g_0)|\,=\,(f_1-g_1)(z_1)+k \leq \frac{1}{2} |f_0-g_0| .
\end{equation}

	 During the iteration procedure this will happen from time to time for
	 $f_n = \mathcal{G}^n (f_0) $ and  $g_n = \mathcal{G}^n (u)=u $. This is  a good explanation for the outstanding performance of the algorithm.

$\hfill \diamondsuit$
	
	 \medskip

	 \begin{definition} \label{deedee} Given a  Lipschitz potential $A$ with a unique subaction $u\in \mathcal{C}$ consider the set $\mathfrak{B}\subset \mathcal{C}$
	 	 of functions $f_0$,
	 	such that,
	 	if  $ |f_0-u|=|(f_0- u )+ \alpha_{f_0-u)}|_0=  (f_0-u) (r)+ \alpha_{f_0-u} $, for some $r$, then,
	 	$$(f_0-u) (r) \neq  (f_0-u) (\tau_1(r))\,\,\text{and}\,\, (f_0-u) (r)\neq (f_0-u) (\tau_2(r)).$$
	 	
	 \end{definition}

The set $\mathfrak{B}$ is dense in $\mathcal{C}$. The proof of this fact  is basically the same as
the proof that $\mathfrak{A}$ is  dense on $\mathcal{C}\times \mathcal{C}$ and will be not presented.
	
\medskip

In the same way as before one can show that:

	 \begin{theorem} \label{tyy}
	 	Given the function $f_0\in \mathcal{C}$,
	 	assume $f_0 \in\mathfrak{B} $.
	 	In this case, if
	 	$ |\mathcal{G}(f_0)- u|= |f_0-u|,$
	 	then, $f_0=u.$ This implies that if $f_0\neq u$, then
	 \begin{equation} \label{iimp} |\mathcal{G}(f_0)- u|< |f_0-u|.
	 \end{equation}
	 	
\medskip

Therefore,  if  $\mathcal{G}^n(f_0)\in\mathfrak{B} $  and $\mathcal{G}^n(f_0)\neq u$, then
\begin{equation} \label{iimpil}  |\mathcal{G}^{n+1}(f_0)- u|< |\mathcal{G}^n(f_0)-u|.
	 	 \end{equation}

\end{theorem}

Given an initial $f_0$ from time to time    $\mathcal{G}^n(f_0)\in\mathfrak{B} $ for some $n$, and
then the next iterate will experience a better approximation to the calibrated subaction $u$.

\medskip

Now we will prove that	$\mathfrak{A}$ is dense.
We will need first to state some  preliminary properties which will be used later. We recall that the norm in $\mathcal{C}$ is given by
$\displaystyle |f|=\inf_{d \in R} |f+d|_{0}$
and the distance in $\mathcal{C} \times \mathcal{C}$ is the max distance
$d((f,g),(f',g')):=\max ( |f-f'|, |g-g'|)$
which is equivalent to the product topology. We will show now that
the set $\mathfrak{A}$ is dense in
$\mathcal{C} \times \mathcal{C}$ with respect to this topology.

Consider $X=[0,1]$ and the maps $\tau_2(x)=\frac{1}{2} x$ and $\tau_2(x) = \frac{1}{2} (x+1)$. Let $\mathcal{F}=\left\{(f,g) |\, \,f, g \,\,\in \mathcal{C}\right\}\subset \mathcal{C} \times \mathcal{C}$.
Denote by $\beta$ the map  $\beta: X\times \mathcal{F} \to \mathbb{R}$ given by
$$\beta(x,f,g)= |f-g|-|f(x)-g(x)| + \min_{i \in \{0,1\}}\left\{|f-g|-|f(\tau_{i}(x))-g(\tau_{i}(x))| \right\}.$$

We notice that $\beta(x,f,g) \geq 0$, and, moreover
\begin{itemize}
	\item $\beta(x,f,g) = 0$, if and only if, $|f-g|=|f(x)-g(x)|$, and, \\ $|f-g|=|f(\tau_1(x))-g(\tau_1(x))|$ or $|f-g|=|f(\tau_2(x))-g(\tau_2(x))|$;
	\item  $\beta(x,f,g) > 0$, if and only if, one of the two conditions is true
	
	$|f-g|>|f(x)-g(x)|$, or,
	\\ $|f-g|>|f(\tau_1(x))-g(\tau_1(x))|$ and $|f-g|>|f(\tau_2(x))-g(\tau_2(x))|$.
\end{itemize}

We define the set $\mathcal{O} \subset \mathcal{F}$ as being
$$\mathcal{O}_{\mathcal{F},\delta}=\left\{(f,g)\in \mathcal{F} | \beta(x,f,g)>0,\; \forall x \in [\delta, 1-\delta]\right\}.$$

If  $d= - \frac{\max (f-g) + \min(f-g)}{2},$
then
$|f-g|= |f- g +d|_0= \frac{\max (f-g) - \min(f-g)}{2}  .$

From the previous observation we conclude that for all $(f,g) \in  \mathcal{O}_{\mathcal{F},\delta}$, if, $x$ is such that
$|f-g+d|=|f(x)-g(x)+d|$, then,
$|f-g+d|\neq |f(\tau_1(x))-g(\tau_1(x))+d|$ and $|f-g+d|\neq |f(\tau_2(x))-g(\tau_2(x))+d|$.

To motivate our proof we are going to consider an explicit example where we made a perturbation of a pair $(f,g) \in \mathcal{C}$, but $\beta(x,f,g)=0$, for some $x$.

\begin{example} \label{yyy} Consider  $(f,g) \in \mathcal{C}$ where
	\[f(x)=\left\{
	\begin{array}{ll}
	\frac{16}{3}\,x -2 & 0\leq x \  and \  x<3/8 \\
	32x^2-36x+9& 3/8\leq x \  and \  x<3/4 \\
	64x^2-104x+42& 3/4\leq x \  and \  x\leq 7/8 \\
	-16x+14& 7/8\leq x \  and \  x\leq 1.
	\end{array}
	\right.
	\]
	and $g(x)=0$.
	\begin{figure}[h!]
		\center
		\includegraphics[scale=0.5, height=2cm]{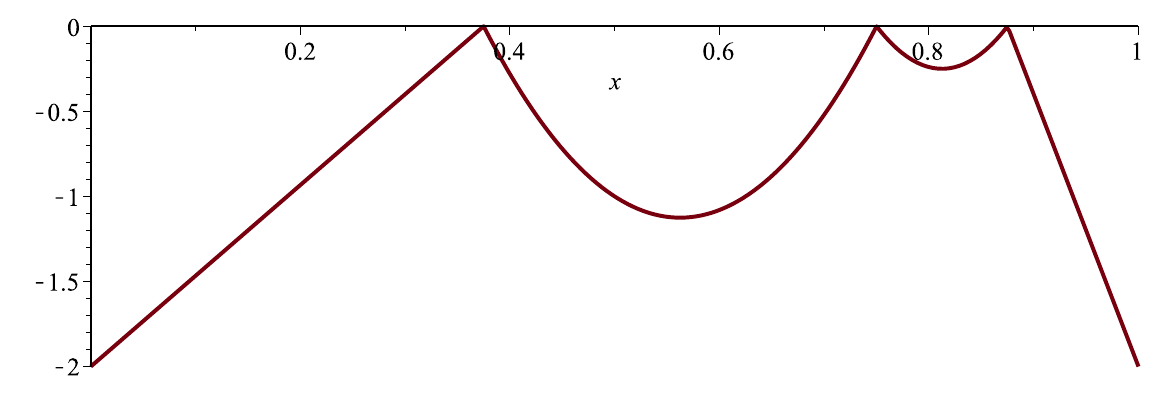}
		\caption{ $f(x)$ of example \ref{yyy} }\label{x16}
	\end{figure}

	It is easy to see that for $x=3/4$ we have
	$|f-0|=|f-0+1|=f\left(3/4\right)+1=1$, $ f\left(\tau_2\left(3/4\right)\right)+1=1$ and $f\left(\tau_1\left(3/4\right)\right)+1=1,$ (see Figure~\ref{x16})
	thus, $\beta\left(3/4, f,0\right)=0$, meaning that $(f,0) \not\in \mathcal{O}_{\mathcal{F},\frac{1}{4}}$. The same is true for $x=0$.
	
	In order to obtain the perturbation $(f_{\varepsilon}, g_{\varepsilon})$ we consider an $\varepsilon$-concentrated approximation via Dirac function $u_{\varepsilon}(x):={\frac {1}{\varepsilon\,\sqrt {\pi }}{{\rm e}^{-{\frac {{x}^{2}}{{\varepsilon}^{2}}}}}}$ (see Figure~\ref{x13})
		
	\begin{figure}[h!]
		\center
		\includegraphics[height=1.5cm, width=5cm]{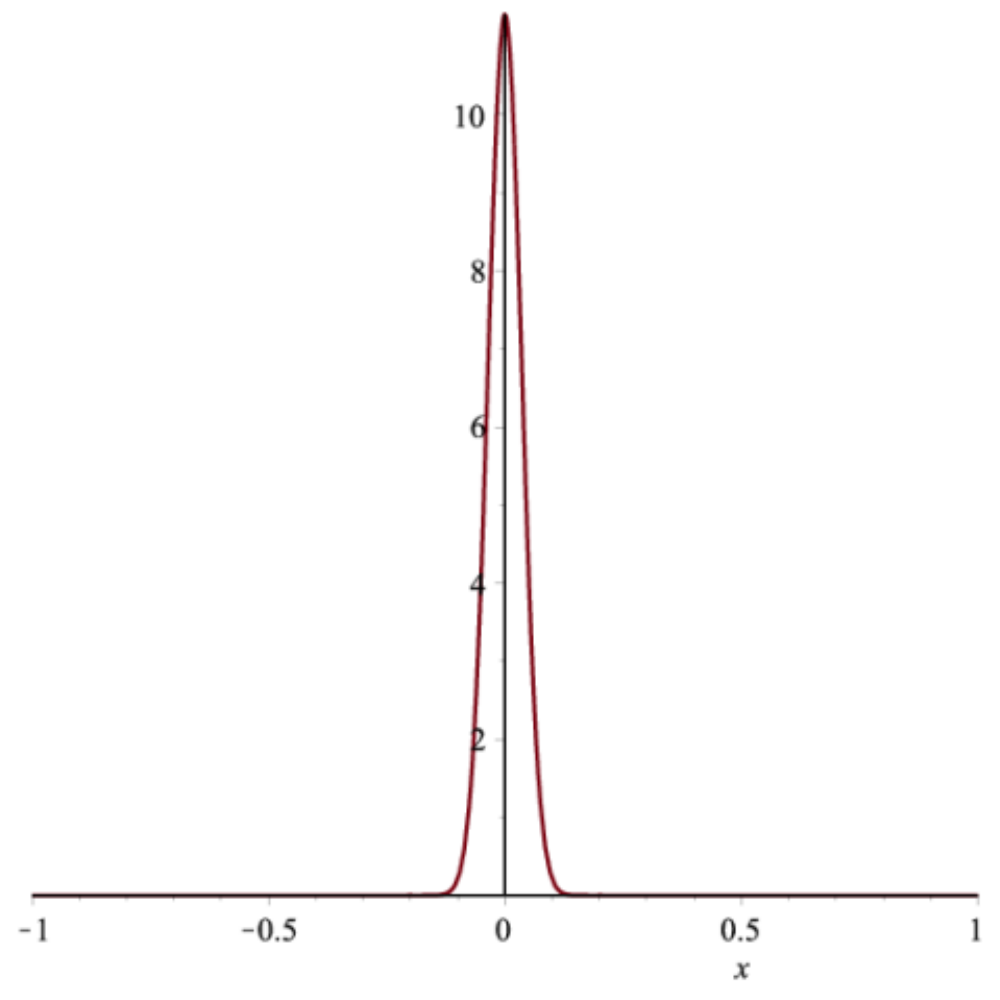}
		\caption{ $u_{\varepsilon}$}\label{x13}
	\end{figure}
	and we also we define for $\varepsilon=0.005$ the modifications (see Figure \ref{x12}):
	$$Q_{\varepsilon}(x):={\frac{1}{500}}u_{\varepsilon} \left( x-(3/4-0.015)\right) \text{ and }
	W_{\varepsilon}(x):=-{\frac{1}{1000}}u_{\varepsilon} \left( x-(0+0.015)\right):$$
	
We set
	$f_{\varepsilon}(x)=f \left( x \right)+Q_{\varepsilon}(x)+W_{\varepsilon}(x) \text{ and } g_{\varepsilon}(x)=g(x).$	
	In this case $|f_{\varepsilon}-f|=|-Q_{\varepsilon}-W_{\varepsilon}|=(0.113-(-0.226))/2=0.1695$ as we can see by the picture (see Figure \ref{x14}).

	As we can see, after the perturbation the maximum value is attained only for $x_0=\frac{3}{4}-0.015$ and for $x_1=0+0.015$ and neither of them are pre-image one of each other. Therefore, $(f_{\varepsilon}, g_{\varepsilon}) \in \mathcal{O}_{\mathcal{F},0}$  (see Figure \ref{x5}).

\end{example}
%\medskip

\begin{theorem}\label{gen_property}
	Let $\Lambda \subset \mathcal{F}$ a compact subset. Then the set $\mathcal{O}_{\Lambda,\delta}$ is an open and dense set. In particular,
	$\mathcal{O}_{\Lambda}:=\bigcap_{n >2} \mathcal{O}_{\Lambda,\frac{1}{n}}$
	is a dense set.
\end{theorem}

\begin{figure}[h!]
		\center
		\includegraphics[height=2cm, width=5cm]{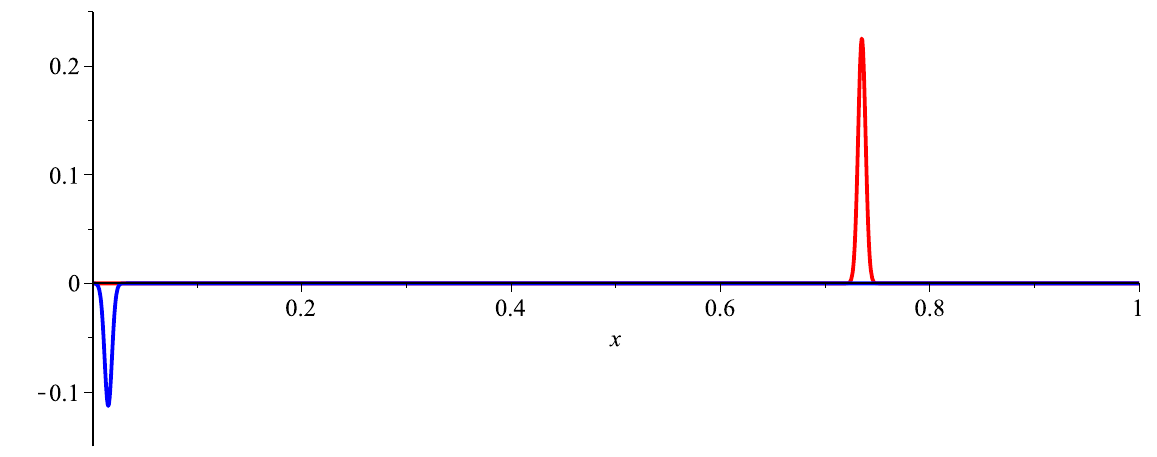}
		\caption{ $Q_{\varepsilon}$ (red) and $W_{\varepsilon}$(blue).}\label{x12}
	\end{figure}
As a consequence, taking $\mathfrak{A}=\mathcal{O}_{\Lambda}$, it will follow:
\begin{corollary} \label{cor} The set $\mathfrak{A}$ is dense. More precisely, if  $|f-g|=|f-g+d|=f(x_0)-g(x_0)+d$,
	then $f(\tau_{i}(x_0))-g(\tau_{i}(x_0))+d \neq f(x_0)-g(x_0)+d$ , for $i=1,2$.
\end{corollary}

\begin{proof}
	The first step in the proof of Theorem~\ref{gen_property} is the openness of
	$\mathcal{O}_{\Lambda,\frac{1}{n}}$.
	
	In this direction we observe that $\beta$ is continuous because the min operation and the sup-norm are continuous. Taking $(f_0,g_0) \in \mathcal{O}_{\Lambda,\frac{1}{n}}$ we obtain
	$\beta(x,f_0,g_0)>0,\; \forall x \in \left[\frac{1}{n}, 1-\frac{1}{n}\right],$ as we can see in the Figure \ref{fig:xx}.

	Using the compactness and the continuity we can take $\alpha>0$, such that, $\beta(x,f_0,g_0)>\alpha,\; \forall x \in [\frac{1}{n}, 1-\frac{1}{n}]$.
	Therefore, if $(f,g) \in \mathcal{U}$, where  $\mathcal{U}$ is   an open neighborhood of $(f_0,g_0)$, we  get
	$$\beta(x,f,g)-\frac{\alpha}{2}=\beta(x,f,g)-\beta(x,f_0,g_0)+\beta(x,f_0,g_0)-\alpha+\alpha -\frac{\alpha}{2}\geq$$
	$$\leq \beta(x,f_0,g_0)-\beta(x,f,g)+\beta(x,f,g)-\alpha+\alpha -\frac{\alpha}{2}> -\varepsilon_{x}+0+\frac{\alpha}{2}>0,$$
	if we choose $\varepsilon_{x}< \frac{\alpha}{2}$, where $\varepsilon_{x}$ is the continuity constant for the map $(f,g) \to \beta(x,f,g)$, for a fixed $x \in \left[\frac{1}{n}, 1-\frac{1}{n}\right]$.

		\begin{figure}[h!]
		\center
		\includegraphics[height=1.5cm, width=5cm]{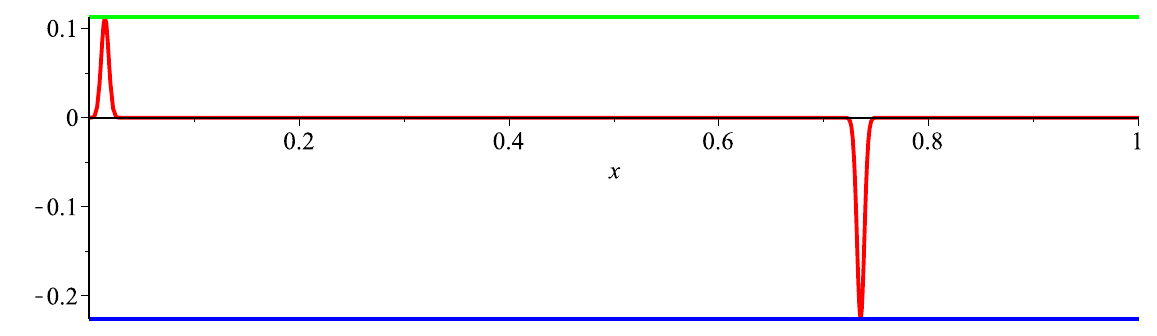}
		\caption{ $Calculating \,\,|-Q_{\varepsilon}-W_{\varepsilon}|$.}\label{x14}
	\end{figure}
	Since the interval $\left[\frac{1}{n}, 1-\frac{1}{n}\right]$ is compact we can take $0<\varepsilon \leq \varepsilon_{x},\; \forall x \in \left[\frac{1}{n}, 1-\frac{1}{n}\right]$.

	\begin{figure}[h!]
		\center
		\includegraphics[height=1.5cm, width=5cm]{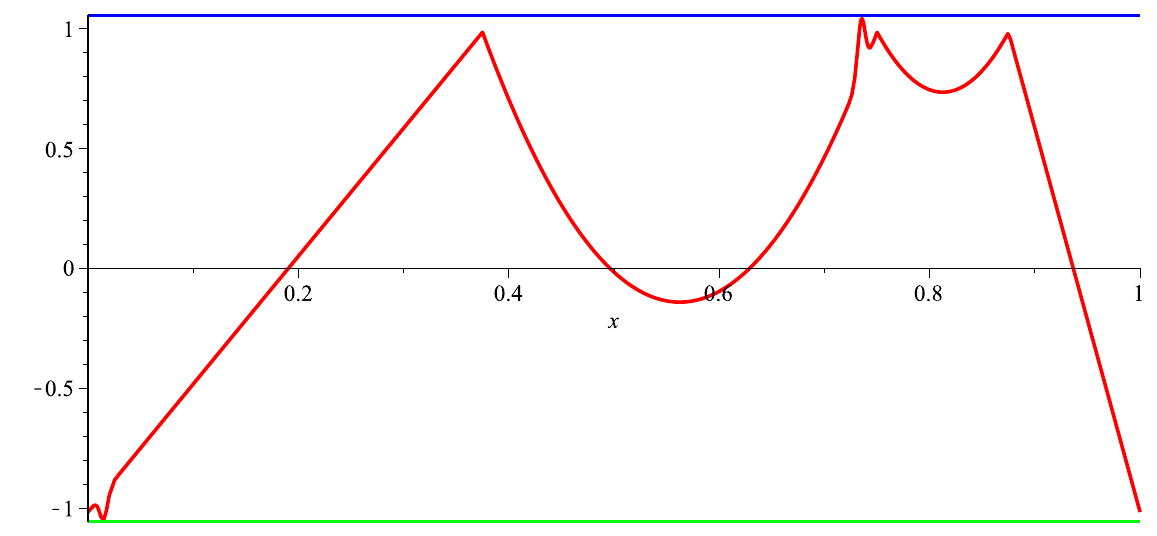}
		\caption{ $|f_{\varepsilon} -g_{\varepsilon}|=|f_{\varepsilon} -g_{\varepsilon} -(-0.985)|_{0}=(0.07-(-2.04))/2=1.055$.}\label{x5}
	\end{figure}

	This proves that the set
	$$\mathcal{U}_{\delta}:=$$
	$$_{\left\{ (f,g) \;|\;\text{ if }d((f,g),(f_0,g_0))< \delta, \text{ then } |\beta(x,f,g)-\beta(x,f_0,g_0)|<\varepsilon, \; \forall x \in \left[\frac{1}{n}, 1-\frac{1}{n}\right]\right\}}$$
	is an open neighborhood of $(f_0,g_0)$ in  $\mathcal{O}_{\Lambda,\frac{1}{n}}$.
	\begin{figure}[h!]
		\centering
		\includegraphics[height=3.3cm, width=13cm,trim=0in 2in 0in 2in, clip]{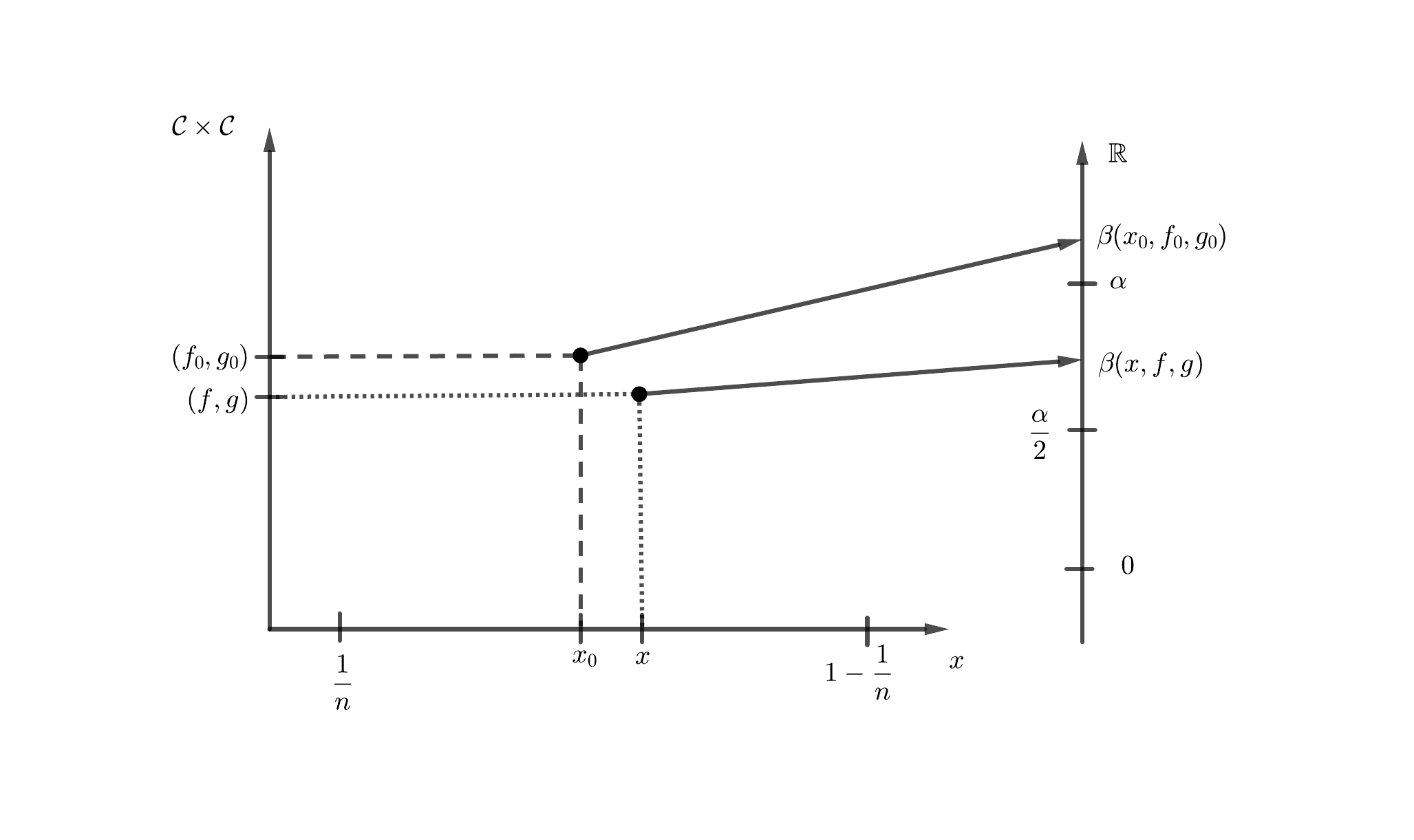}
		\caption{ Approximating $(x_0, f_0, g_0)$.}
		\label{fig:xx}
	\end{figure}

	In order to prove the density of $\mathcal{O}_{\Lambda,\frac{1}{n}}$ we observe that if $ x_0 \in \left[\frac{1}{n}, 1-\frac{1}{n}\right]$, then
	$\frac{1}{2n}+\frac{i}{2}-x_0\leq \tau_{i}(x_0) - x_0\leq \frac{i}{2}+\frac{1}{2}-\frac{1}{2n}-x_0.$
	Thus $|\tau_{i}(x_0) - x_0|\geq \frac{1}{2n}$ for all  $ x_0 \in [\frac{1}{n}, 1-\frac{1}{n}]$.
	
	Using this estimate we can apply an $\varepsilon$-concentrated perturbations with  $\varepsilon<\frac{1}{2n}$ (see Example \ref{yyy} for a constructive approach) obtaining  a pair $(f_{\varepsilon},g_{\varepsilon})$, in such way that, $g_{\varepsilon}=g$, $x_0$ and $x_1$ are the only points where $|f_{\varepsilon}-g_{\varepsilon}|=|f_{\varepsilon}(x_0)-g(x_0)+d|=
	|f_{\varepsilon}(x_1)-g(x_1)+d|$ and  $x_0 \neq \tau_0(x_1), \tau_2(x_1)$, $x_1 \neq \tau_0(x_0), \tau_2(x_0)$.
	
	In particular  $\beta(x,f_{\varepsilon},g)>0$, for any $ x \in \left[\frac{1}{n}, 1-\frac{1}{n}\right]$,
	which means that $(f_{\varepsilon},g_{\varepsilon})\in \mathcal{O}_{\Lambda,\frac{1}{n}}$.
\end{proof}

\medskip

\section{Perturbation theory:   close by the fixed point} \label{close}

In this section we analyze the question: when the calibrated subaction is unique is there an uniform exponential speed of  approximation   of the iteration $\mathcal{G}^n (f_0)$ to the subaction? The question makes sense close by the subaction $u$. The answer is no.
We will proceed a careful analysis of the action of $\mathcal{G}$ close by the fixed point $u\in \mathcal{C}.$

Section~\ref{close} is about the possibility of change a given point, in the neighborhood of a subaction, by a close one having different properties, with respect to the convergence rate of the operator $\mathcal{G}$.  It can't be used for genericity, as far as we know because we say nothing close to other points in the space.

In some examples we may consider a different dynamical system  on $X=[0,1]$ given by the maps $\tau_{i}(x)=\frac{1}{2}(i+1-x)$, for $i=0,1$, which are the inverse branches of $T(x)=-2x \mod 1$.

Our main task is to evaluate the effect of a perturbation on the nonlinear operator $\psi$ defined by
$$\psi(f)(x)= \max_{T(y)=x} (A+f)(y)= \max_{i=0,1} (A+f)(\tau_{i}(x))$$
for a fixed potential $A \in \mathcal{C}_k$.

\begin{figure}[h!] %trim= left bottom right top
  \centering
  \includegraphics[scale=0.18,trim=0.5in 4in 0.5in 4in, clip]{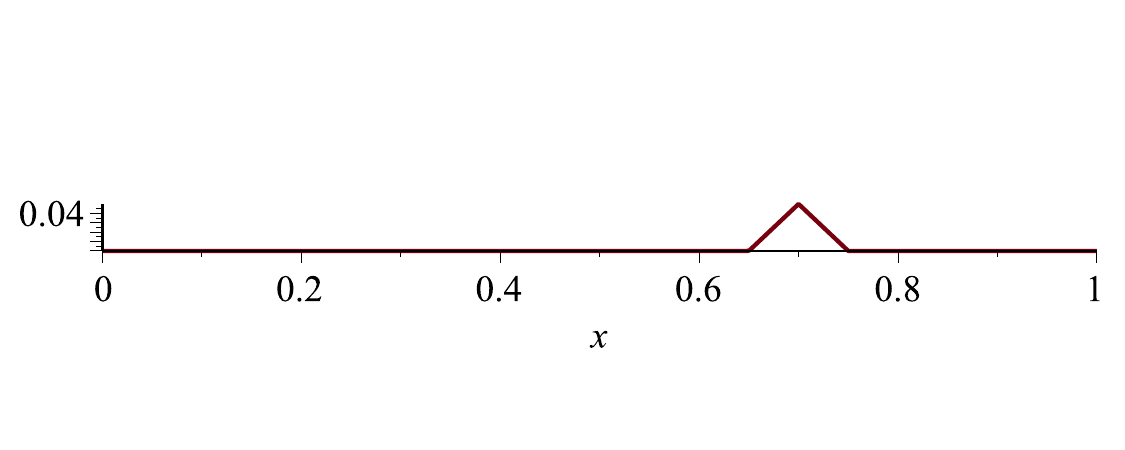}\qquad
  \includegraphics[scale=0.18,trim=0.5in 4in 0.5in 4in, clip]{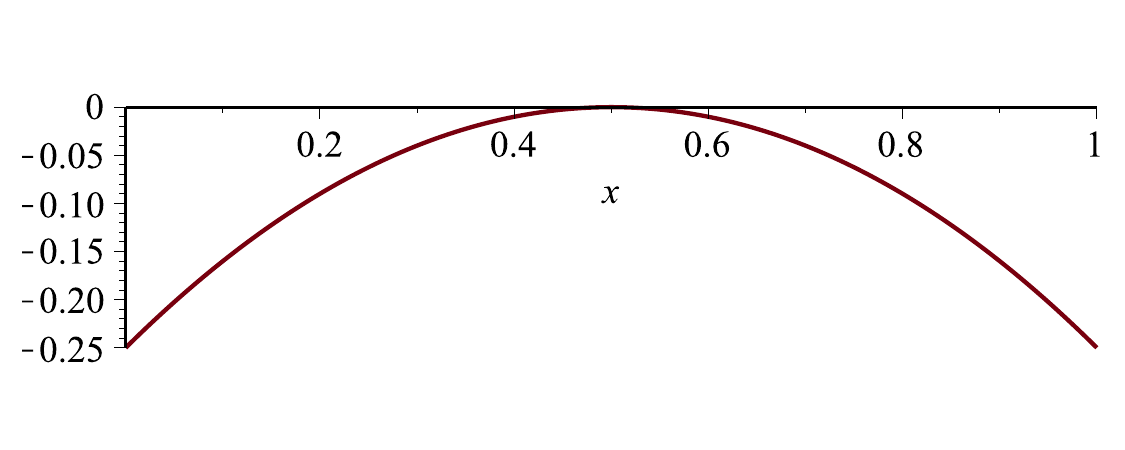}\qquad
  \includegraphics[scale=0.18,trim=0.5in 4in 0.5in 4in, clip]{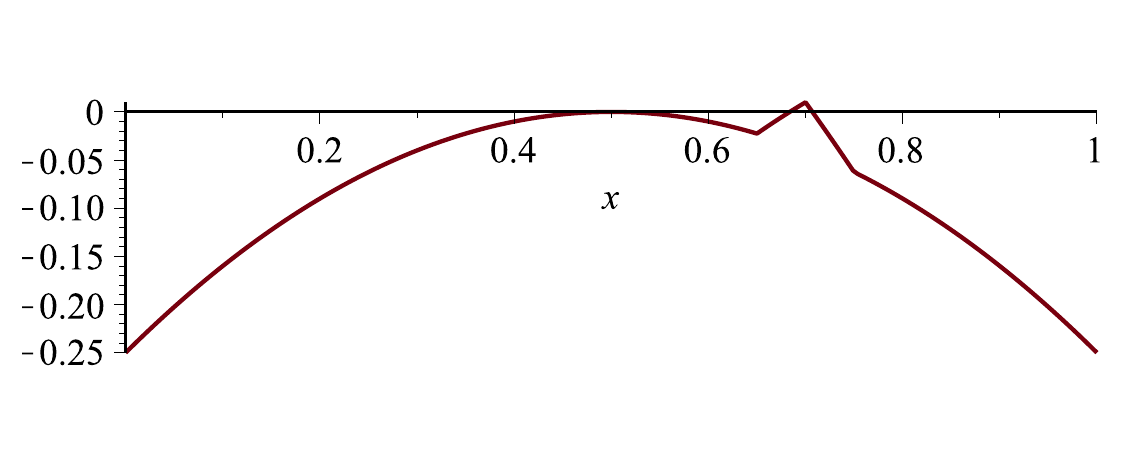}
  \caption{The graph of the function $\alpha_{0.05,0.7}$ in the left side, $f(x)=-(x-1/2)^2$ in the center  and $f_{0.05}=f(x)+ \alpha_{0.05,0.7}(x)$ in the right side.}\label{x1}
\end{figure}

The operator $H:=H_{A}$ given by
$$H(f)(x):=\frac{1}{2} f(x) + \frac{1}{2}\psi(f)(x),$$

Note that $\mathcal{G}:=\mathcal{G}_{A}$  is a normalized version of $H$
$$\mathcal{G}(f)(x):=H(f)(x) - \sup_{x \in X}H(f)(x).$$
It is usual to denote $\displaystyle c_{f}:=\sup_{x \in X}H(f)(x)$ then $H(f)(x)=\mathcal{G}(f) + c_{f}$ (normalization means that $\displaystyle \sup_{x \in X}\mathcal{G}(f)(x)=0$). Sometimes it is useful to look at the operator $H-Id$ given by
$(H-Id)(f):=\frac{1}{2}\psi(f)(x) - \frac{1}{2} f(x).$

We assume that there exists \textbf{a unique function $u \in \mathcal{C}_K$} such that $\mathcal{G}(u)=u$ (this is true if the maximizing probability is unique). Thus, $H(u)(x)=\mathcal{G}(u) + c_{u}=u(x)+ c_{u}$ where $\displaystyle c_{u}:=\sup_{x \in X}H(u)(x)$.

The above equation is equivalent to
$u(x)+c_{u} =\frac{1}{2} u(x) + \frac{1}{2}\psi(f)(x)$
which is equivalent to the sub-action equation
$$u(x)=\max_{i=0,1} (A-2c_{u} + u)(\tau_{i}(x)).$$

We can assume that $m_{A}=2c_{u}=0$ (by adding a constant to $A$) and then, $H(u)=u$.
It is useful to observe that under this assumption we also get   $\psi(u)=u$.

We start with a local perturbation lemma.

Let $\alpha_{\varepsilon,a}:X \to \mathbb{R}$ be a piecewise linear bump function defined by
\[ \alpha_{\varepsilon,a}(x)=
\left\{
  \begin{array}{ll}
    0, & 0 \leq x \leq a-\varepsilon \\
    kx-k(a-\varepsilon), & a-\varepsilon \leq x \leq a \\
    -kx+k(a+\varepsilon), & a\leq x \leq a+\varepsilon \\
    0, & a+\varepsilon \leq x \leq 1 ,\\
  \end{array}
\right.
\]
where $a \in (0,1)$ and $\varepsilon >0$ is arbitrary small.

\begin{lemma}\label{L1}
  If $f \in \mathcal{C}_K$,  then $f_{\varepsilon}=f(x)+ \alpha_{\varepsilon,a}(x) \in \mathcal{C}_K$. Moreover, $f_{\varepsilon}(x) \geq f(x)$ and $f_{\varepsilon}(x) = f(x)$ outside of the interval $[a-\varepsilon , \; a+\varepsilon]$. Finally, $|f_{\varepsilon} - f| = \frac{k\varepsilon}{2}$.
\end{lemma}
\begin{proof}
The proof is straightforward because $|f_{\varepsilon} - f|= |\alpha_{\varepsilon,a}|$ and $0 \leq \alpha_{\varepsilon,a}(x) \leq k\varepsilon$.
\end{proof}

We will make the perturbations by choosing a fixed point $x_0 \neq 0, 1, 1/2$  in $X$  and $\varepsilon>0$, such that, the intervals $I=[x_0-\varepsilon , \; x_0+\varepsilon]$ and $T(I)=[T(x_0)-2\varepsilon , \; T(x_0)+ 2\varepsilon]$ are disjoint. Then, we take $f_{\varepsilon}=f(x)+ \alpha_{\varepsilon,a}(x)$ and we will try to estimate $\psi(f_{\varepsilon})$.\\

\begin{lemma}\label{L2}
  $\psi(f_{\varepsilon})=\psi(f)$ outside of $T(I)$.
\end{lemma}
\begin{proof}
\begin{figure}[h!]
  \centering
  \includegraphics[width=4cm, height=2.2cm,trim=0in 1.5in 0in 1.5in, clip]{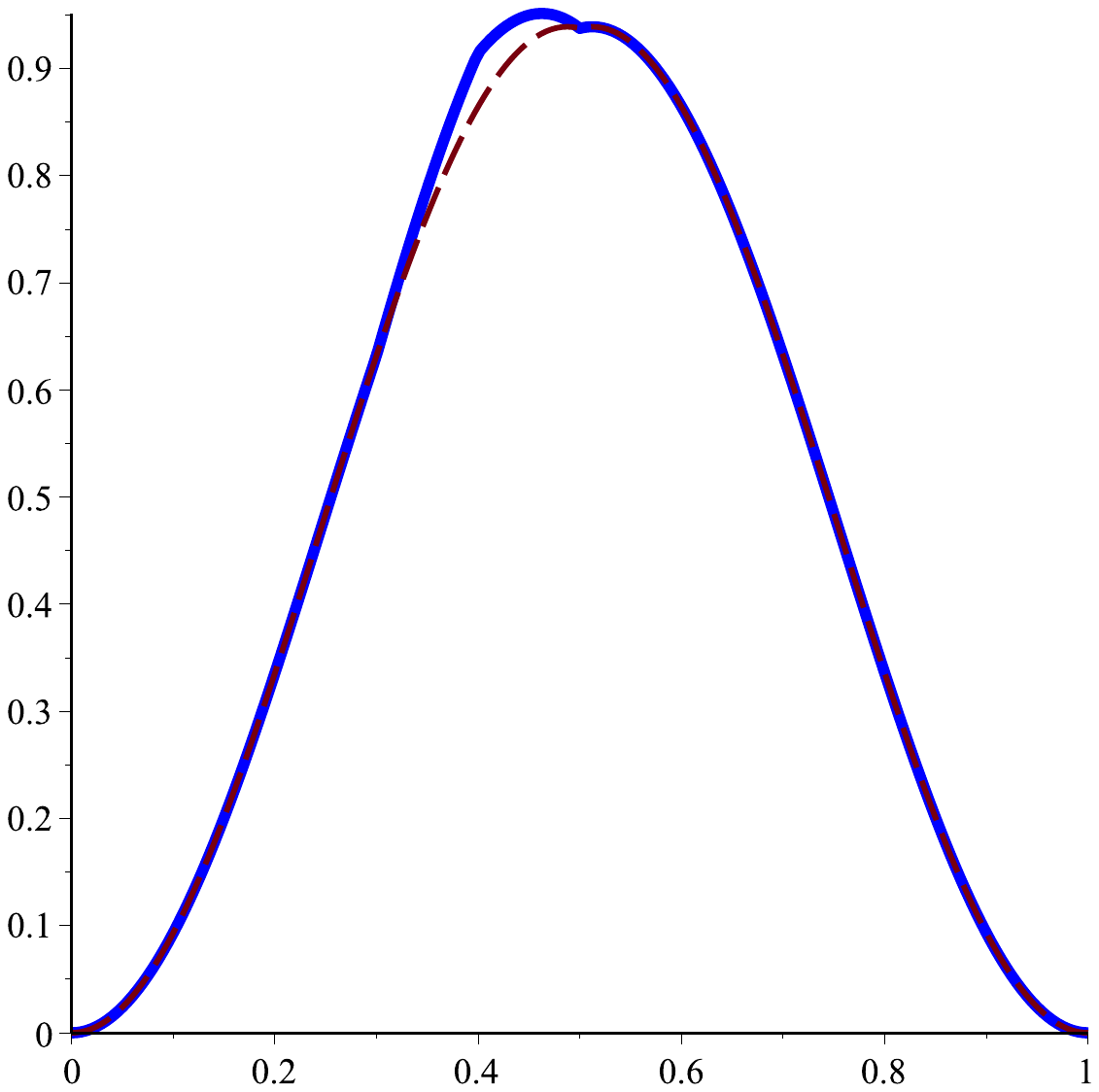}
  \caption{The graph of the functions $\psi(f_{\varepsilon})$ (blue line) and $\psi(f)$ (traced line) where, $A(x)=\sin^2(2\pi x)$, $f(x)=-(x-1/2)^2$  and $f_{0.1}=f(x)+ \alpha_{0.1,0.7}(x)$. The difference occurs only in the interval $T(I)=[0.2, \; 0.6]$ because $T(0.7)=0.4$ and $I=[0.6, \; 0.8]$.}\label{x2}
\end{figure}

  We notice that $A$ remains unchanged and $[T(x_0)-2\varepsilon , \; T(x_0)+ 2\varepsilon]=T([x_0-\varepsilon , \; x_0+\varepsilon])$. Therefore, for any $y$ such that $T(y)=x$ we can not have $y \in [x_0-\varepsilon , \; x_0+\varepsilon]$. Thus, $f_{\varepsilon}(y)=f(y)$, proving that $\psi(f_{\varepsilon})=\psi(f)$.
\end{proof}

Another question is about what happens in $T(I)$. For any $x$ in this interval one of its pre-images $y$ belongs to $I$ therefore $f_{\varepsilon}(y) \geq f(y)$. Thus, $\psi$ may change.

We recall that \textbf{a turning point $x$} (see also \cite{LOS} and \cite{LOT}) is a point where $(A+f)(\tau_1(x))=(A+f)(\tau_2(x))$. If $x$ is not a turning point then there exists \textbf{a dominant realizer}, that is,  $(A+f)(\tau_1(x))>(A+f)(\tau_2(x))$, or, $(A+f)(\tau_1(x))<(A+f)(\tau_2(x))$.

\begin{lemma}\label{L3}
  Suppose that $x_0$  is such that $T(x_0)$ is not a turning point and $j$ is the dominant symbol. Let $i \in \{0, \; 1\}$ be such that $\tau_{i}( T(x_0))= x_0$. We have two possible cases:
  \begin{itemize}
    \item If $j \neq i$, then $\psi(f_{\varepsilon})(x)=\psi(f)(x)$, for any $ x \in T(I)$.
    \item If $j = i$, then $\psi(f_{\varepsilon})(x)=\psi(f)(x)+ \alpha_{\varepsilon,x_0}(\tau_{j}(x)) \geq \psi(f)(x)$, for any $ x \in T(I)$ and $|\psi(f_{\varepsilon})(x) - \psi(f)(x)| = \frac{k\varepsilon}{2}$.
  \end{itemize}
\end{lemma}
\begin{proof}
  In the first case, in order to fix ideas we suppose, without lost of generality, $j=0$ and $i=1$, then $\tau_2( T(x_0))= x_0$ and $(A+f)(\tau_1(T(x_0)))>(A+f)(\tau_2(T(x_0)))$. By the continuity of $A+f$ we can choose $\varepsilon>0$ small enough in order to have $(A+f_{\varepsilon})(\tau_1(x))>(A+f_{\varepsilon})(\tau_2(x))$, for all $ x \in T(I)$. Therefore, $\psi(f_{\varepsilon})(x) = (A+f_{\varepsilon})(\tau_1(x))=(A+f)(\tau_1(x))= \psi(f)(x)$, for any $ x \in T(I)$.\\
  In the second case, $\tau_{j}( T(x_0))= x_0$ and $(A+f)(\tau_{j}(T(x_0)))>(A+f)(\tau_{i}(T(x_0)))$. Once more we use the continuity of $A+f$ to choose $\varepsilon>0$ small enough in order to have $(A+f_{\varepsilon})(\tau_{j}(x))>(A+f_{\varepsilon})(\tau_{i}(x))$, for all $ x \in T(I)$. Therefore, $\psi(f_{\varepsilon})(x) = (A+f_{\varepsilon})(\tau_{j}(x))=(A+f)(\tau_{j}(x))+ \alpha_{\varepsilon,x_0}(\tau_{j}(x))= \psi(f)(x)+ \alpha_{\varepsilon,x_0}(\tau_{j}(x))$, for any $ x \in T(I)$.
\end{proof}

Our first task is to compare $H(f)$ and $H(f_{\varepsilon})$. We can always assume that $T(I)$ and $I$ are disjoint so the perturbation $f \to f_{\varepsilon}$ acts separately in each one as described by the previous lemmas.

\begin{lemma} \label{L4}
  Let $f_{\varepsilon}$ a perturbation of $f$ and $x_0$ such that is not a pre-image of a turning point (with respect to $f$). Then, $H(f)(x) \leq H(f_{\varepsilon})(x)$,  with equality only outside of $[T(x_0)-2\varepsilon , \; T(x_0)+ 2\varepsilon] \cup [x_0-\varepsilon , \; x_0+\varepsilon]$. Moreover, $ H(f_{\varepsilon})(x)- H(f)(x)\leq \frac{k\varepsilon}{2} $. (We can prove similar results for $(H-Id)$.)
\end{lemma}
The proof is a direct consequence of the previous lemmas.
\begin{figure}[h!]
  \centering
  \includegraphics[width=5.0cm, height=2.5cm,trim=0.5in 3in 0.5in 3in, clip]{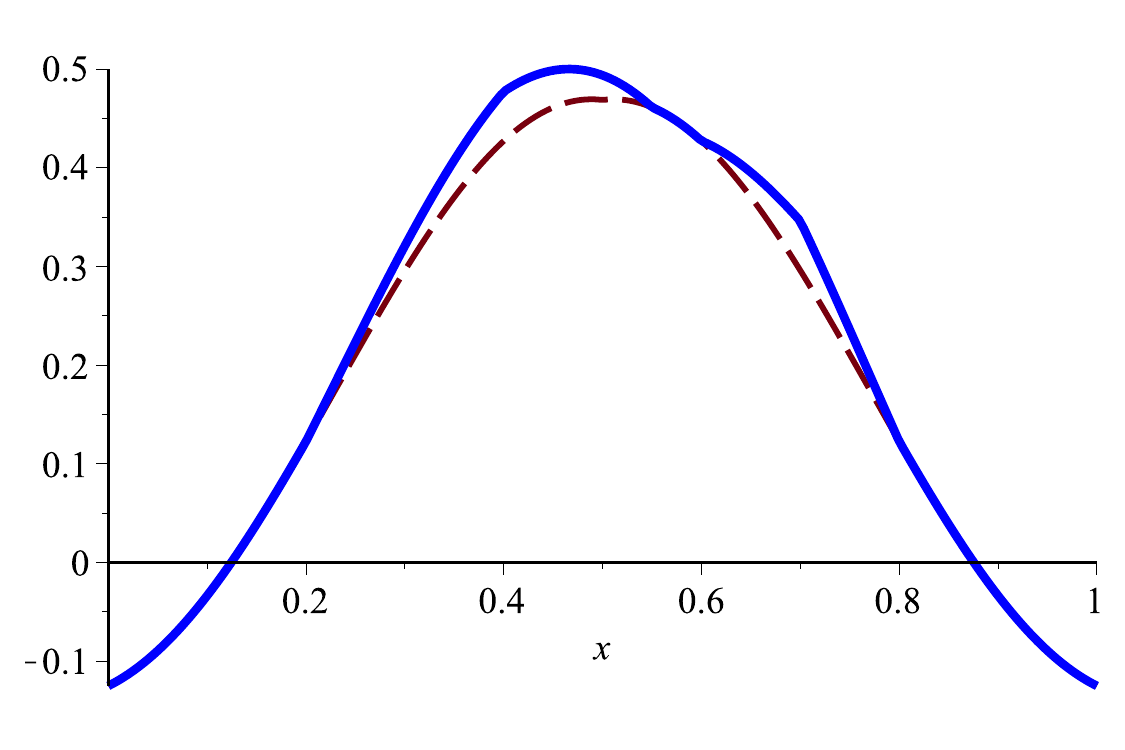}
  \caption{ The graph of the functions $H(f_{\varepsilon})$ (blue line) and $H(f)$ (traced line) where, $A(x)=\sin^2(2\pi x)$, $f(x)=-(x-1/2)^2$  and $f_{0.1}=f(x)+ \alpha_{0.1,0.7}(x)$. The difference occurs only in the interval $[0.2, \; 0.6] \cup [0.6,\; 0.8]$ because $T(0.7)=0.4$.}\label{x3}
\end{figure}

We want to study the relation between $|\mathcal{G}(f) - u|$ and $|f-u|$. We also want to see what happens when we make a perturbation $f \to f_{\varepsilon}$.\\

We start by choosing $d=\alpha_{f-u}$ such that $\delta=|f-u|=|f-u+d|_0$, then
$$-\delta \leq f(y)-u(y)+d \leq \delta,$$
for all $y \in X$.
Multiplying the above by $1/2$ we conclude that
$$-\frac{\delta}{2}   \leq \frac{1}{2}(f(y) -u(y))  +\frac{d}{2} \leq \frac{\delta}{2}.$$

Adding $A(y)$ we obtain the inequalities
$$-\delta \leq A(y)+f(y)-(A(y)+ u(y))+d \leq \delta,$$
and
$$-\delta + (A(y)+ u(y)) \leq A(y)+f(y) +d \leq \delta + (A(y)+ u(y)).$$
Taking the supremum in $y$, such that, $T(y)=x$, we get
$-\delta +\psi(u)(x) \leq \psi(f)(x) +d \leq \delta + \psi(u)(x).$
Multiplying by $1/2$ we conclude that
$$-\frac{\delta}{2}   \leq \frac{1}{2}(\psi(f)(x) -\psi(u)(x))  +\frac{d}{2} \leq \frac{\delta}{2}.$$

Note that
$$\mathcal{G}(f)(x)- u(x) +d = \mathcal{G}(f)(x) - \mathcal{G}(u)(x) +d =$$
 $$=\frac{1}{2} (f(x)-u(x)) +  \frac{1}{2} (\psi(f)(x)-\psi(u)(x)) -c_{f} + c_{u}+d=$$
 $$\frac{1}{2} (f(x)-u(x)+d) +  \frac{1}{2} (\psi(f)(x)-\psi(u)(x)+d) - c_{f}. $$

Using the inequalities $$-\frac{\delta}{2}   \leq \frac{1}{2}(\psi(f)(x) -\psi(u)(x))  +\frac{d}{2} \leq \frac{\delta}{2},$$
$$-\frac{\delta}{2}   \leq \frac{1}{2}(f(y) -u(y))  +\frac{d}{2} \leq \frac{\delta}{2},$$ and, the fact that $c_{u}=0$, we finally obtain
$$-\frac{\delta}{2}-\frac{\delta}{2} -c_{f} \leq \mathcal{G}(f)(x)- u(x) +d  \leq  \frac{\delta}{2}+\frac{\delta}{2}-c_{f},$$
and,
$$-\delta  \leq \mathcal{G}(f)(x)- u(x) +(d +c_{f}) \leq  \delta.$$
Therefore,
$$|\mathcal{G}(f)(x)- u(x) +(d +c_{f})| \leq \delta= |f-u|,$$
for all $x \in X$.

From this fundamental inequality we get a very important result about the operator $\mathcal{G}$.

  We recall that $|\mathcal{G}(f)(x)- u(x)| = \min_{\gamma} |\mathcal{G}(f)- u +\gamma |_{0} \leq |\mathcal{G}(f)- u + (d +c_{f})  |_{0}= \sup_{x \in X} |\mathcal{G}(f)(x)- u(x) +(d +c_{f})| \leq  |f-u|$.\\

\begin{theorem}\label{T1}
  Let $\mathcal{G}$ be the operator associated to $A$ and $u$ the fixed point ($\mathcal{G}(u)(x)= u(x)$), then,
  \begin{itemize}
    \item[a)] The contraction rate is controlled by $H -Id$;
    \item[b)] $|H(f) - f|_{0} \leq 2 |f-u|$;
    \item[c)]If $|H(f) - f|_{0} = \beta $, then
    $|\mathcal{G}(f)(x)- u(x) +(d +c_{f})|_{0} \geq  |f-u| -\beta$.
  \end{itemize}
\end{theorem}
\begin{proof}

  (a) We recall that $\mathcal{G}(f)(x)+c_{f}=H(f)$, thus,
  $$|\mathcal{G}(f)(x)- u(x) +(d +c_{f})| \leq  |f-u|$$
  $$|\mathcal{G}(f)(x)+c_{f} -f(x) + f(x)- u(x) +d| \leq  \sup_{x \in X}|f(x)-u(x)+d|$$
  $$|[H(f) -f(x)] + f(x)- u(x) +d| \leq  \sup_{x \in X}|f(x)-u(x)+d|$$
  $$\sup_{x \in X} |[H(f) -f(x)] + f(x)- u(x) +d| \leq  \sup_{x \in X}|f(x)-u(x)+d|.$$
  (b) Here we use the triangular inequality
  $$|H(f) -f(x)|\leq |[H(f) -f(x)] + f(x)- u(x) +d|+ |f(x)-u(x)+d| \leq 2 |f-u|.$$
  (c)  Using the triangular inequality we obtain
  $$|f-u|= |f-u+d|_{0}\leq |f-u+d + \mathcal{G}(f)(x)+c_{f} -f(x) - (\mathcal{G}(f)(x)+c_{f} -f(x))|_{0}\leq$$
   $$\leq |\mathcal{G}(f)(x)+c_{f} -f(x)+ f-u+d|_{0}+| \mathcal{G}(f)(x)+c_{f} -f(x)|_{0} =$$
    $$=|\mathcal{G}(f)(x)-u + (d+c_{f})|_{0}+| H(f)(x) -f(x)|_{0}= |\mathcal{G}(f)(x)-u + (d+c_{f})|_{0}+\beta,$$
   or, equivalently,
   $$|\mathcal{G}(f)(x)-u + (d+c_{f})|_{0} \geq |f-u| -\beta.$$
\end{proof}

\begin{figure}[h!]
  \centering
  \includegraphics[width=4cm, height=2.5cm,trim=0.5in 3.5in 0.5in 4in, clip]{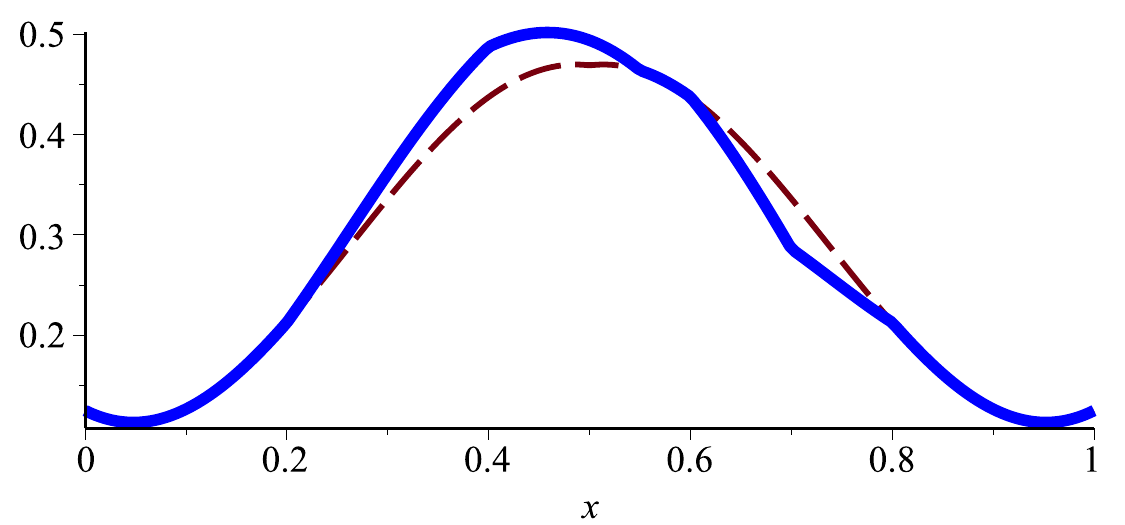}
  \caption{ Functions $(H-Id)(f_{\varepsilon})$ (blue line) and $(H-Id)(f)$ (traced line) where, $A(x)=\sin^2(2\pi x)$, $f(x)=-(x-1/2)^2$  and $f_{0.1}=f(x)+ \alpha_{0.1,0.7}(x)$. The difference occurs only in the interval $[0.2, \; 0.6]$, where the perturbation is bigger, and,  the interval $ [0.6,\; 0.8]$, where the perturbation is smaller, because $T(0.7)=0.4$.}\label{x4}
\end{figure}

We are dealing with a kind of technical problem: $| p(x)+ q(x)| \leq |q|_0, \; \forall x \in X$, where $\max q= -\min q$. In our case, $p(x)= H(f) -f(x)$ and $q(x)=  f(x)- u(x) +d$ are continuous functions. The first observation is that $| p(x)+ q(x)| \leq |q|_0, \; \forall x \in X$, is equivalent to $-|q|_0 - q(x)\leq  p(x) \leq |q|_0 - q(x)$. From this we can get interesting examples.

\begin{example}\label{E1}
  Consider $p(x)= -4\, \left( x-1/2 \right) ^{2}$ and $q(x)=\cos \left( 2\,\pi \,x \right)$. It is easy to see that $|q|_0=\max q= -\min q=1$ and the inequality $-1 - q(x)\leq  p(x) \leq 1 - q(x)$ is described in the Figure \ref{x5}.
  \begin{figure}[h!]
  \centering
  \includegraphics[width=4cm, height=3cm,trim=0.5in 1.5in 0.5in 1.6in, clip]{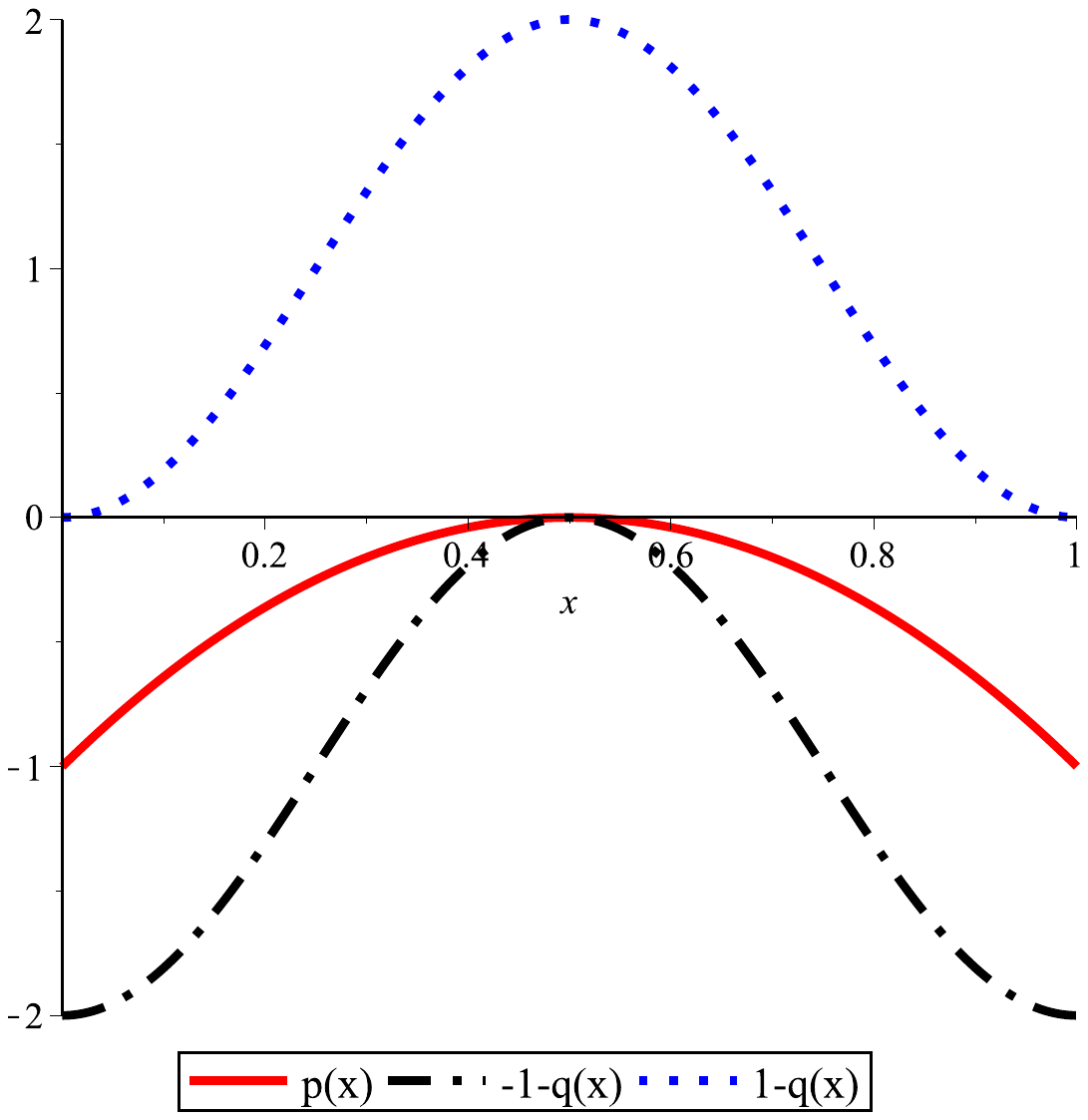}
  \caption{ Functions $-1-q(x)$ and $1-q(x)$.}\label{x5}
\end{figure}
A simple calculation shows that $|p+q|_0=1=|q|_0$, but $|p+q|= |p+q+0.414|_0=0.586$.

The property $\max q= -\min q$ means that $|q|=|q+0|_0$, therefore, $|p+q|=0.586<1=|q|$.
\end{example}

\begin{lemma} \label{L5}
  Consider $| p(x)+ q(x)| \leq |q|_0, \; \forall x \in X$, with $\max q= -\min q$. Then, there exists $z \in X$, such that, $p(z)=0$. In particular, taking $p(x)= H(f) -f(x)$ and $q(x)=  f(x)- u(x) +d$, we have
  $$f(z)=\max_{T(y)=z} A(y) + f(y).$$
\end{lemma}
\begin{proof}
  We already know that there exists $x_0$ such that $|q|_0=q(x_0)$, therefore,
  $p(x_0)+ q(x_0) \leq |q|_0=q(x_0)$, or, equivalently,   $p(x_0) \leq 0.$ Analogously, there exists $x_1$ such that $|q|_0=-q(x_1)$ and $p(x_1) \geq 0.$ Unless $q=cte$ we can always suppose that $x_0 \neq x_1$. If $p(x_0)= 0$ or $p(x_1)= 0$ the problem is solved. Otherwise, if $p(x_0)< 0$ and $p(x_1)> 0$ the intermediate value theorem for continuous functions claims that there exists $z \in [x_0, \; x_1]$, such that, $p(z)=0$.

  Note that for $p(x)= H(f) -f(x)$, the equation $p(z)=0$ is equivalent to $\displaystyle f(z)=\max_{T(y)=z} A(y) + f(y).$
\end{proof}

The behaviour of $|\mathcal{G}(f)(x)-u + (d+c_{f})|_{0}$ may be very different from $|\mathcal{G}(f)-u |$. On the one hand $|\mathcal{G}(f)-u | \leq |\mathcal{G}(f)-u + (d+c_{f})|_{0} \leq |f-u|$ and on the other hand we can find $f$ arbitrarily close to $u$, such that, $|\mathcal{G}(f)-u |= \frac{1}{4} \leq |f-u|$.

\begin{figure}[h!]
  \centering
  \includegraphics[width=4cm, height=1.5cm,trim=0.5in 4in 0.5in 4in, clip]{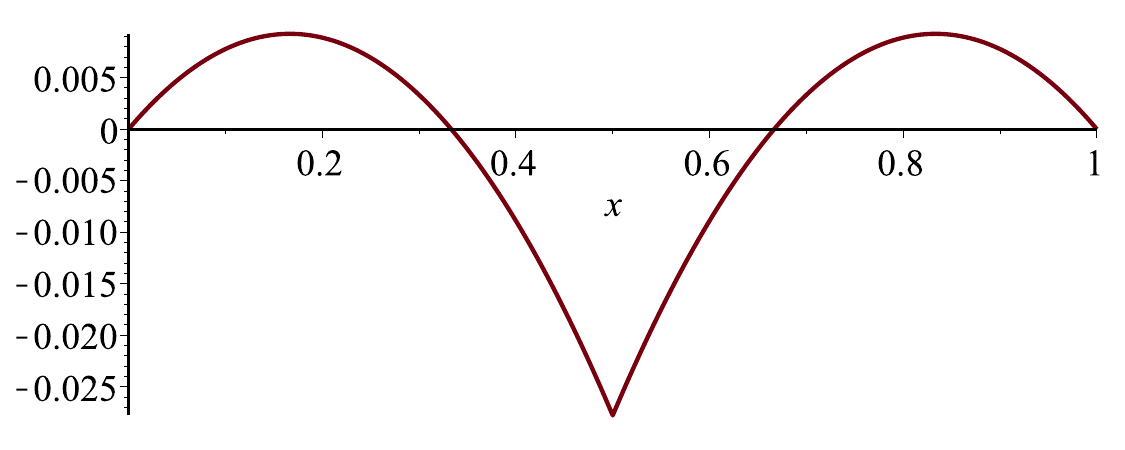} \qquad \includegraphics[width=4cm, height=1.5cm,trim=0.5in 4in 0.5in 4in, clip]{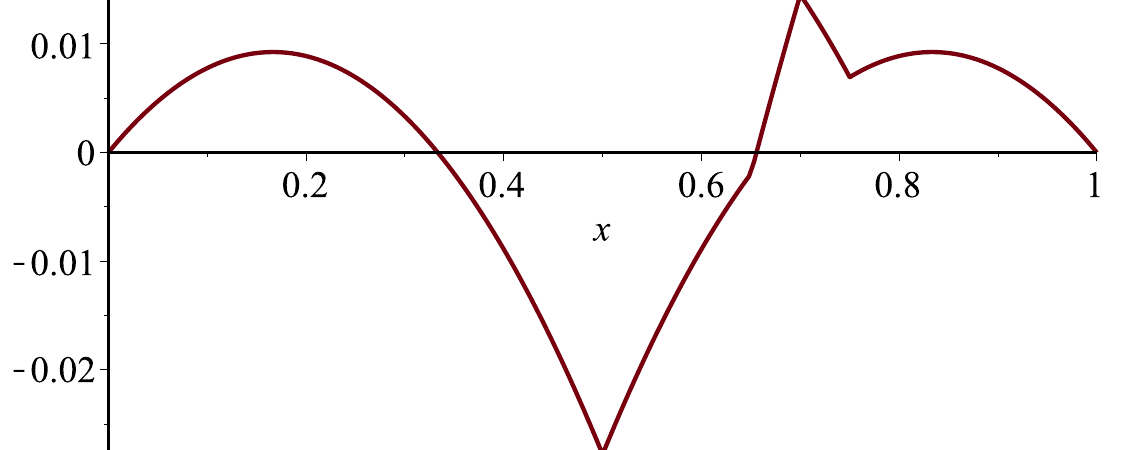}
  \caption{ In the left side the graph of $u$ and in the right side the graph of   $f_{\varepsilon}$.}\label{x6}
\end{figure}

\begin{lemma} \label{L6}
  Let $u$ be the only sub-action of $A$ ($m_A=0$). Let $f_{\varepsilon}=u + \alpha_{\varepsilon,x_0}$ a perturbation of $f$ and take $x_0$ not a pre-image of a turning point (with respect to $f$). Then, $|\mathcal{G}(f_{\varepsilon})-u |= \frac{1}{2}  |f_{\varepsilon}-u|$ and $|f_{\varepsilon}-u|=\frac{k \varepsilon}{2}$.
\end{lemma}
\begin{proof}
  First, we observe that $|f_{\varepsilon}-u|=|\alpha_{\varepsilon,x_0}|=\frac{ \max \alpha_{\varepsilon,x_0}-\min \alpha_{\varepsilon,x_0}}{2}=\frac{k \varepsilon-0}{2}=\frac{k \varepsilon}{2}$.

  Rewriting  $|\mathcal{G}(f_{\varepsilon}) -u |$ we obtain
  $$|\mathcal{G}(f_{\varepsilon}) -u |=|H(f_{\varepsilon})-c_{f_{\varepsilon}}-u|=|H(f_{\varepsilon})-u|=|\frac{1}{2} f_{\varepsilon} + \frac{1}{2}\psi(f_{\varepsilon})-u|=$$
  $$=|\frac{1}{2} (u + \alpha_{\varepsilon,x_0}) + \frac{1}{2}\psi(f_{\varepsilon})- \psi(u)|=|\frac{1}{2} \alpha_{\varepsilon,x_0} + \frac{1}{2}(\psi(f_{\varepsilon})- \psi(u))|.$$
  The function $\alpha_{\varepsilon,x_0}$ is zero outside of the set $[x_0-\varepsilon , \; x_0+\varepsilon]$, and, $\psi(f_{\varepsilon})- \psi(u)=0$ outside of the set $[T(x_0)-2\varepsilon , \; T(x_0)+ 2\varepsilon]$ by Lemma~\ref{L2}.

  \begin{figure}[h!]
  \centering
  \includegraphics[width=5cm, height=2.5cm,trim=0.5in 3in 0.5in 3in, clip]{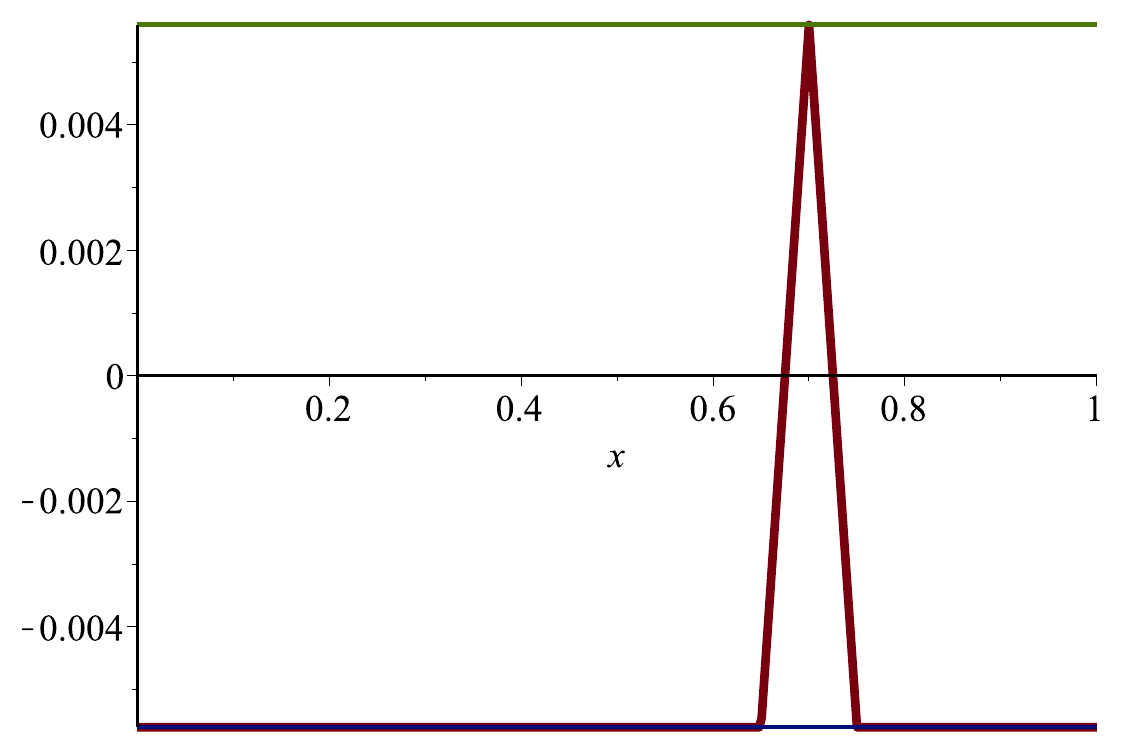} \qquad \includegraphics[width=5cm, height=2cm,trim=0.5in 4in 0.5in 4in, clip]{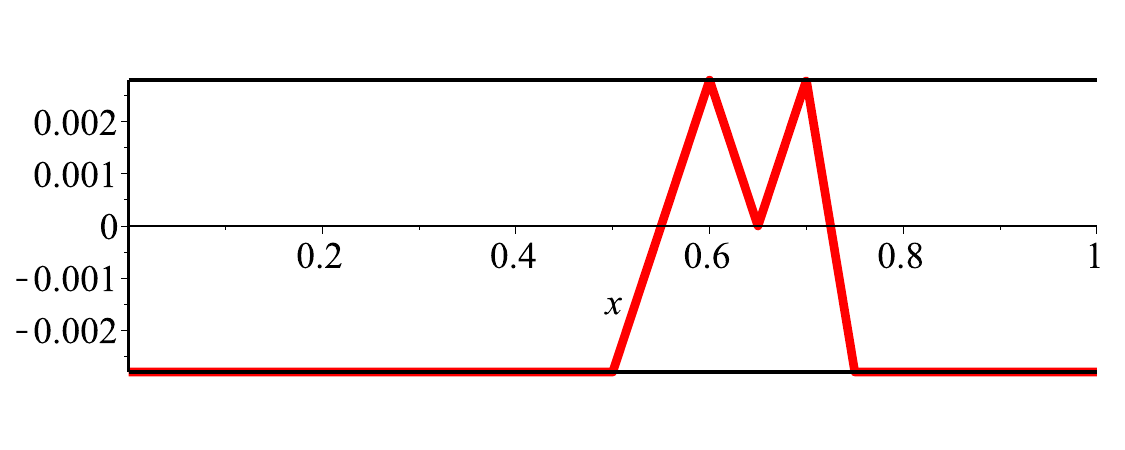}
  \caption{ In the left  the graph of $f_{\varepsilon} \left( x \right) -u \left( x \right) -\frac{k\varepsilon}{2}$ and in the right the one for  $1/2\,f_{\varepsilon} \left( x \right) +1/2\,\psi \left( f_{\varepsilon} \right)(x) -u \left( x \right) -\frac{k\varepsilon}{4} $.}\label{x7}
\end{figure}

  Therefore, the  $\min \frac{1}{2} \alpha_{\varepsilon,x_0} + \frac{1}{2}(\psi(f_{\varepsilon})- \psi(u))=0$, and, $\max \frac{1}{2} \alpha_{\varepsilon,x_0} + \frac{1}{2}(\psi(f_{\varepsilon})- \psi(u)) = \frac{k \varepsilon}{2}$. By definition $|\mathcal{G}(f_{\varepsilon})(x)-u |= \frac{k \varepsilon}{4}$.
\end{proof}

\begin{example}\label{E2}
   Consider the dynamics $T(x)=-2x \mod 1$.

   Let $A(x)=-(x-\frac{1}{2})^2+ \frac{1}{36}$ be the potential and  $u$ the subaction (see Figures \ref{x6}, \ref{x7} and \ref{x9})
   \[ u(x)= \left\{
       \begin{array}{ll}
         -1/3\,{x}^{2}+x/9,          & 0 \leq x  \leq 1/2 \\
         -1/3\,{x}^{2}+5/9\,x-2/9,  & 1/2\leq x\leq 1.
       \end{array}
     \right.
   \]

\begin{figure}[h!]
  \centering
  \includegraphics[width=5cm, height=3cm,trim=0.5in 2.8in 0.5in 2.8in, clip]{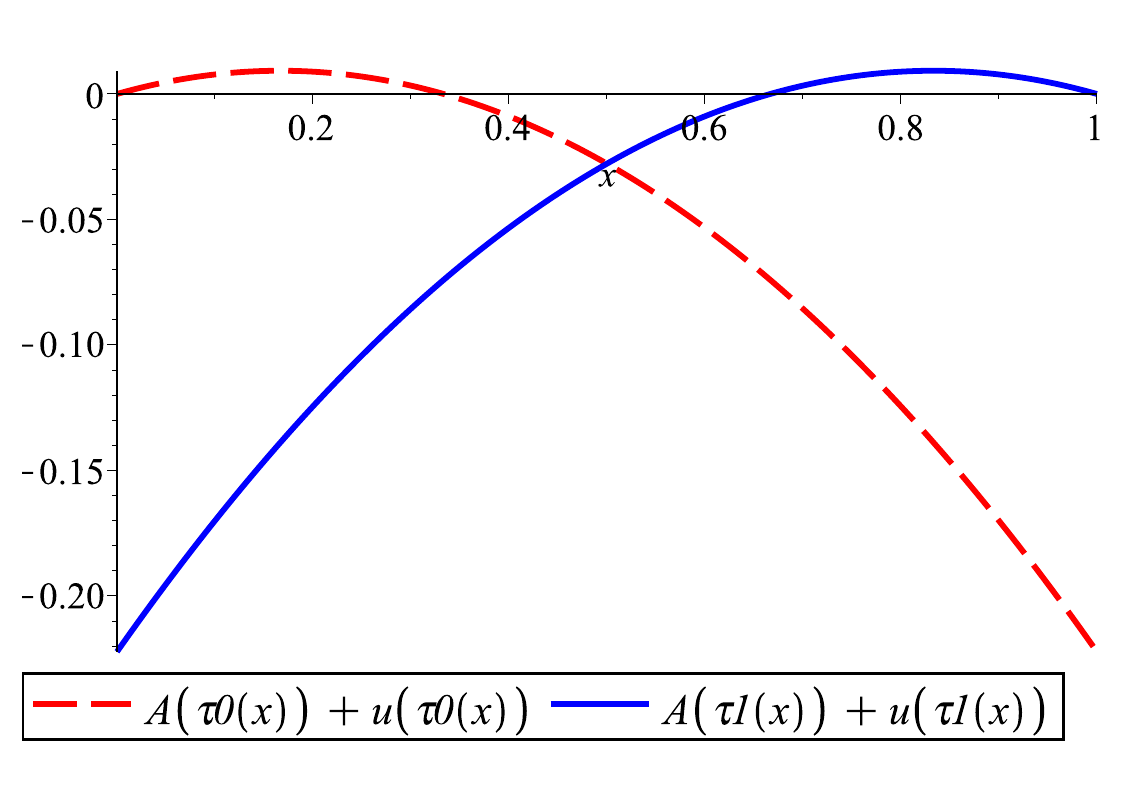}
  \caption{The graph of the functions $(A+u)(\tau_1(x))$ and $(A+u)(\tau_2(x))$.}\label{x9}
\end{figure}

From the graph of $u$ we see that $x=\frac{1}{2}$ is the only turning point. Therefore,  we can take $x_0=0.7$, $\varepsilon = 0.05$ and $f_{\varepsilon}=u + \alpha_{0.05,0.7}$. We also know that $Lip(A)=1$ and $Lip(u)=\frac{2}{9}$, thus, we can take $k=\frac{2}{9}$.

As predicted $|\mathcal{G}(f_{\varepsilon})(x)-u |= \frac{k \varepsilon}{4}=0.0028$ and $|f_{\varepsilon}-u |= \frac{k \varepsilon}{2}=0.0056$.
\end{example}

From Lemma~\ref{L6} we get
\begin{corollary} \label{ftp}
  For any $\varepsilon>0$ there exists a function $f$ which is $\varepsilon$-close to $u$, such that, $\mathcal{G}$ contracts by $1/2$ in $f$, that is, $|\mathcal{G}(f)-u |= \frac{1}{2}  |f-u|$.
\end{corollary}

We may ask if there exists some neighborhood of $u$ where $|\mathcal{G}(f_{\varepsilon})(x)-u | \leq (1-\delta) |f_{\varepsilon}-u |$. The answer is no. Actually, it is the opposite of that. We can exhibit a sequence $f_{\varepsilon} \to u$, and, $|\mathcal{G}(f_{\varepsilon})(x)-u | = |f_{\varepsilon}-u |$.

\begin{example} \label{E3} We will show an example where $|\mathcal{G}(f_{\varepsilon}) -\mathcal{G}(u )| =      |f_{\varepsilon} -u |$ , $\epsilon>0$, for $f_{\varepsilon}$  as close as you want to the calibrated subaction $u$.

Consider again the dynamics $T(x)=-2x$ \,(mod 1). Let $A(x)=-(x-\frac{1}{2})^2+ \frac{1}{36}$ be the potential and $u$ the subaction
   \[ u(x)= \left\{
       \begin{array}{ll}
         -1/3\,{x}^{2}+x/9,          & 0 \leq x  \leq 1/2 \\
         -1/3\,{x}^{2}+5/9\,x-2/9,  & 1/2\leq x\leq 1.
       \end{array}
     \right.
   \]

We fix $x_0=\frac{2}{3}$. The function $\alpha_{\varepsilon,x_0}$ is zero outside of $I=[\frac{2}{3}-\varepsilon , \; \frac{2}{3}+\varepsilon]$ and $\psi(f_{\varepsilon})- \psi(u)=0$ outside of $T(I)$ by Lemma~\ref{L2}.

We know that $T(\frac{1}{3})=\frac{1}{3}$ and $T(\frac{2}{3})=\frac{2}{3}$.
As we can see in the Figure~\ref{x9}, $\{0,\; \frac{1}{2} , \;1\}$, are the only turning points and the dominant symbol in $x_0=2/3$ is $j=1$. Also, $\tau_2( T(\frac{2}{3}))= \frac{2}{3}$, and thus $i=1 =j$.

Once more $$|\mathcal{G}(f_{\varepsilon}) -u |=\left|\frac{1}{2} \alpha_{\varepsilon,x_0} + \frac{1}{2}(\psi(f_{\varepsilon})- \psi(u))\right|.$$

Since $I \subset T(I)$, we get, by Lemma~\ref{L3}, that $\alpha_{\varepsilon,x_0}$ attains the value $k\varepsilon$ and $\psi(f_{\varepsilon})(x)- \psi(u)(x)=\alpha_{\varepsilon,x_0}(\tau_2(x))$ attains the value $k\varepsilon$ at least in $x_0$. Thus, $\frac{1}{2} \alpha_{\varepsilon,x_0} + \frac{1}{2}(\psi(f_{\varepsilon})- \psi(u))$ attains the value $\frac{k\varepsilon}{2} =|f_{\varepsilon} -u |$ (see Figure~\ref{x8} for $\varepsilon=0.01$ and $x_0=\frac{2}{3}$).

Therefore,     $|\mathcal{G}(f_{\varepsilon}) -u | \geq         |f_{\varepsilon} -u |.$

\begin{figure}[h!]
  \centering
  \includegraphics[width=5cm, height=2cm,trim=0.5in 4in 0.5in 4in, clip]{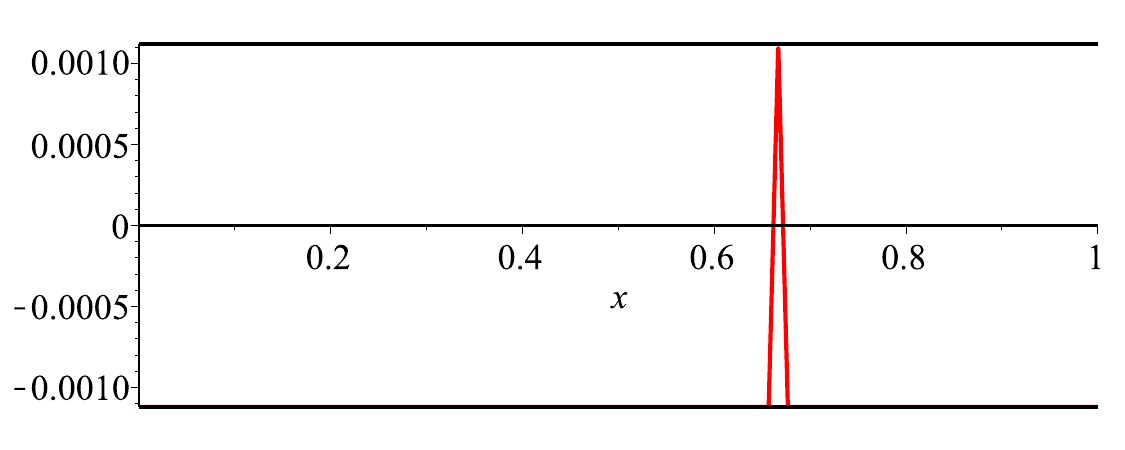} \qquad \includegraphics[width=5cm, height=2cm,trim=0.5in 4in 0.5in 4in, clip]{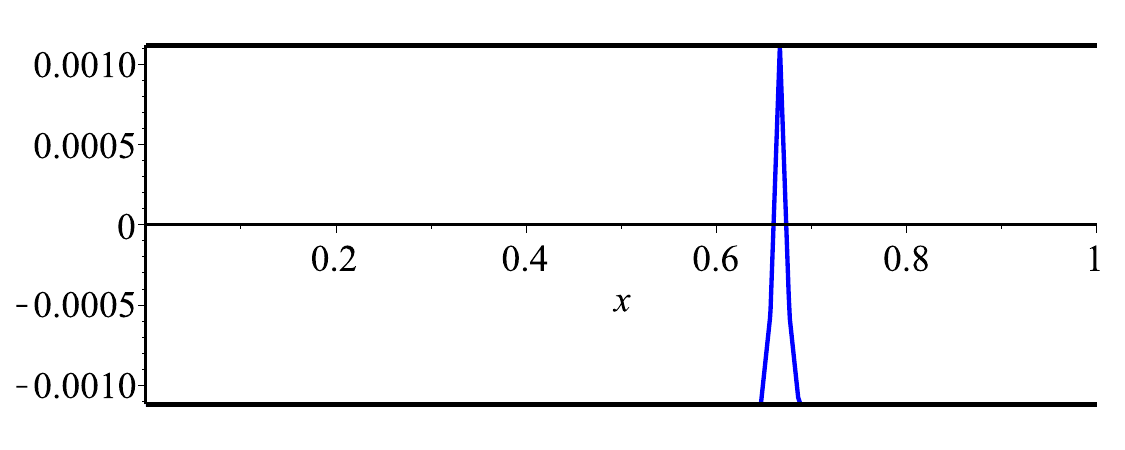}
  \caption{ In the left side the graph of $f_{\varepsilon} \left( x \right) -u \left( x \right) -\frac{k\varepsilon}{2}$ and in the right  side the graph of $1/2\,f_{\varepsilon} \left( x \right) +1/2\,\psi \left( f_{\varepsilon} \right)(x) -u \left( x \right) -\frac{k\varepsilon}{2}$.}\label{x8}
\end{figure}
\end{example}

\end{document}